\documentclass[a4paper,12pt]{article}

\usepackage[utf8]{inputenc}

\usepackage{fontenc}

\everymath{\displaystyle}
\usepackage{amssymb}
\usepackage{amsmath}
\usepackage{amsthm}
\usepackage{listings}
\usepackage{verbatim}
\usepackage[nottoc,numbib]{tocbibind}
\usepackage[toc,page]{appendix}
\usepackage[hidelinks]{hyperref}
\usepackage{geometry} 
\usepackage{romannum}
\usepackage{mathrsfs}
\usepackage{float}
\usepackage{enumitem}
\usepackage{graphicx}
\usepackage[absolute]{textpos}
\usepackage{color}
\usepackage{bbm}
\usepackage[english]{babel}
\allowdisplaybreaks
\usepackage{setspace}
\usepackage{booktabs,colortbl,xcolor}
\usepackage{graphicx,wrapfig}
\usepackage[section]{placeins} 
\usepackage{float} 
\usepackage{enumitem}
\usepackage{subfiles}
\usepackage{scrhack}
\usepackage{tikz}
\usetikzlibrary{calc}
\usepackage{esint}

\usepackage[backend=bibtex, style=numeric, giveninits=true]{biblatex}
\renewbibmacro{in:}{}
\newtheorem{theorem}{Theorem}[section]
\newtheorem{corollary}{Corollary}[theorem]
\newtheorem{lemma}[theorem]{Lemma}
\theoremstyle{definition}

\theoremstyle{definition}
\newtheorem{definition}[theorem]{Definition}

\setlist[enumerate,1]{label=\arabic*.}
\setlist[enumerate,2]{label=\theenumi\arabic*.}
\setlist[enumerate,3]{label=\theenumii\arabic*.}

\def\Xint#1{\mathchoice
	{\XXint\displaystyle\textstyle{#1}}%
	{\XXint\textstyle\scriptstyle{#1}}%
	{\XXint\scriptstyle\scriptscriptstyle{#1}}%
	{\XXint\scriptscriptstyle\scriptscriptstyle{#1}}%
	\!\int}
\def\XXint#1#2#3{{\setbox0=\hbox{$#1{#2#3}{\int}$ }
		\vcenter{\hbox{$#2#3$ }}\kern-.568\wd0}}

\def\dashint{\Xint-}

\date{}

\begin{document}
	\pagenumbering{arabic}	
	\title{Variational Estimates for nonnegative Harmonic Functions}
	\author{Jakob Fromherz}
	\maketitle
	\begin{abstract}
		Let $D\subset\mathbb R^d$ be a bounded domain, $u$ a nonnegative harmonic function on $D$ and $p$ the Poisson kernel of $D$. In this paper we establish boundedness of Jones' variational integral
		\begin{align*}
			\int_D|\nabla u(z)|p(z,\xi)\mathrm dz
		\end{align*}
		on an ultradense, $d-1$ dimensional subset of the boundary of $C^{1,\mathrm{Dini}}$ domains. This result was previously only known for the unit ball $B(0,1)\subset\mathbb R^d$.\\\\
		
		
	\end{abstract}
	
	\section{Introduction} 
	Using complex variable harmonic analysis techniques, Bourgain \cite{Bourgain931}\cite{Bourgain932} showed that any function $f\in H^\infty(\mathbb D)$ has finite radial integral on a dense subset of $\mathbb T$ with dimension 1. As a corollary, any bounded real-valued harmonic function on $B(0,1)\subset\mathbb R^2$, has finite \textit{radial variation}, i.e.
	\begin{align}\label{eq:RadInt}
		\int_0^1|\nabla u(r\theta)|\mathrm dr<\infty
	\end{align}
	for $|\theta|=1$ in a set of dimension 1. Geometrically, \eqref{eq:RadInt}  implies that the curve following $u$ on the ray connecting the origin to $\theta$ has finite length. A real-variable approach that extends this result to $C^2$ domains in $\mathbb{R}^d, d\geq 3$ was given by Havin and Mozolyako \cite{HavinMozol2016}.\par 
	Let $p$ denote the Poisson kernel for the unit ball $B(0,1)\subset\mathbb R^d$. Peter W. Jones was the first to consider estimates for
	\begin{align}\label{eq:FirstFullVarEstimate}
		V(\theta,u):=\int_{B(0,1)}|\nabla u(z)|p(z,\theta)\mathrm dz,
	\end{align}
	and in 2003 conjectured its boundedness for some/many  $\theta\in\mathbb S^{d-1}$ . We thus refer to $V$ as \textit{Jones' variational integral}. Following the terminology introduced by Havin and Mozolyako, a point $\theta\in\mathbb{S}^{d-1}$ that satisfies $V(\theta,u)<\infty$ is called a \textit{Bourgain} point of $u$. In \cite{MuellerRiegler20Bloch}, Müller and Riegler proved the existence of Bourgain points for Jones' variational integral.
	Their method aims to extend the techniques developed by Havin and Mozolyako in \cite{HavinMozol2016} to Lipschitz domains. Nevertheless, to obtain boundedness of $V$, the authors in \cite{MuellerRiegler20Bloch} require symmetry of the Poisson integral, originating in its convolution structure on the unit ball.\par
	
	The purpose of this paper is to break free from the group structure on the unit ball and prove boundedness of Jones' variational integral on large subsets of the boundary of $C^{1,\mathrm{Dini}}$ domains. Simultaneously, this paper presents structural simplifications of Havin and Mozolyako's technique, thus allowing for a broader range of its application. Moreover, we introduce a conical analogue to $V$ and derive its boundedness on large subsets of Lipschitz boundaries. Along the way, we exploit sharp estimates for the Green function.\par 
	Let $U\subset V\subset\mathbb R^d$. Following Havin and Mozolyako, we say that $U$ is \textit{ultradense} in $V$ if $\dim U\cap B=\dim V$ for any ball $B$ whose center lies on $V$. All the domains that are treated in the sequel are assumed to be bounded.
	
	\begin{theorem}\label{thm:MainThm1}
		If $D\subset\mathbb R^d$ is a $C^{1,\mathrm{Dini}}$ domain and $u\geq0$ is harmonic on $D$, there exists $\xi\in\partial D$ such that 
		\begin{align}\label{eq:FullVar}
			V(\xi,u)=\int_D|\nabla u(z)|p(z,\xi)\mathrm{d}z<\infty
		\end{align}
		In fact, the set $\mathcal B(u, D):=\lbrace\xi\in\partial D: V(\xi,u)<\infty\rbrace$ is ultradense in $\partial D$, i.e. for every ball $B$ that has its center on $\partial D$, $\dim \mathcal{B}(u,D)\cap B=d-1$.	
	\end{theorem}
 	$\mathcal{B}(u, D)$ denotes the set of \textit{Bourgain} points of $u$. Our methods allow us to estimate a variant of Jones' variational integral on Lipschitz domains $D$ that are star-shaped with respect to the origin. Let $\omega$ denote harmonic measure on the boundary $\partial D$ and $k$ the Martin kernel with joint pole at the origin (see §\ref{sec:prel}). Consider the measure $\mathrm d\mu:=\mathrm d\omega\mathrm dt$ on $D$ that corresponds to decomposing $D$ in radial and angular coordinates. By that we mean writing $D$ as $\left\lbrace (1-t)\xi:\xi\in\partial D,t\in(0,1]\right\rbrace$ and setting
	\begin{align*}
		\int_D f(z)\mathrm d\mu(z)=\int_0^1\int_{\partial D} f((1-t)\xi)\mathrm d\omega(\xi)\mathrm dt
	\end{align*}
	for functions $f$ on $D$. Here, Jones' variational integral takes the form
	\begin{align}\label{eq:FullVarLip}
		V_{\mathrm{Lip}}(\xi,u):=\int_D|\nabla u(z)|\frac{G((1-\delta(z))\xi,0)}{G(z,0)}k(z,\xi)\mathrm{d}\mu(z).
	\end{align}
	\begin{theorem}\label{thm:MainThmLip}
		Let $D$ be a Lipschitz domain that is star-shaped with respect to the origin and let $u\geq0$ be harmonic on $D$. There exists $\xi\in\partial D$ such that 
		\begin{align*}
			V_{\mathrm{Lip}}(u,\xi)<\infty
		\end{align*}
		In fact, the set $\mathcal B'(u, D):=\lbrace\xi\in\partial D: V_{\mathrm{Lip}}(u,\xi)<\infty\rbrace$ is dense in $\partial D$ and there exists $\gamma\in(0,1]$ such that $\dim(B\cap\mathcal B'(u,D))\geq d-2+\gamma$ for any ball $B$ with center on $\partial D$.
	\end{theorem}
	We point out that the statement of Theorem \ref{thm:MainThmLip} is consistent with the statement of Theorem \ref{thm:MainThm1} in the following sense: When specializing $D$ to be star-shaped $C^{1,\mathrm{Dini}}$, 
	\begin{align*}
		V(\xi,u)<\infty\Leftrightarrow V_{\mathrm{Lip}}(\xi,u)<\infty
	\end{align*}
 	holds. We will later see that by Theorem \ref{thm:BogdanLD}, indeed the factor $G((1-\delta(z))\xi,0)/G(z,0)$ appearing in the integrand of \eqref{eq:FullVarLip} is bounded above and below by positive constants. Moreover, the same is true for the Radon Nikodym derivative of $\omega$ with respect to surface measure and the quotient $k(z,\xi)/p(z,\xi)$, see lemma \ref{lem:harmeasVShausdorff}. The coarea formula then implies that the Radon Nikodym derivative of $\mu$ with respect to Lebesgue measure satisfies uniform two-sided estimates.
 	\par Already on Lipschitz domains in the plane, the quotient $G((1-\delta(z))\xi,0)/ G(z,0)$ displays some interesting behaviour: Let $z$ lie on $(1-t)\partial D$, then its size is determined by
 	\begin{align}\label{eq:InterestingQuotient}
 		\frac{G((1-t)\xi,0)}{G((1-t)\zeta,0)},\quad\zeta\in\partial D.
 	\end{align}
 	Assume that the domain is locally given as an epigraph in the usual coordinate system determined by $\lbrace\vec{e_1},\vec{e_2}\rbrace$. If  $\zeta$ lies on a corner that opens up at a small angle around $\vec{e_2}$, then $G((1-t)\zeta,0)$ behaves like $t^{1+N}$, where $N>0$ depends on the opening angle. On the other hand, if the corner opens up at a small angle around $-\vec{e_2}$, then $G((1-t)\zeta,0)\sim t^{1-\theta}$, $\theta>1/2$ depending on the opening angle. Thus the quotient \eqref{eq:InterestingQuotient} may behave very irregularly, depending on the location of $\xi$ and $\zeta$. See Bass \cite{Bass}.
	\\\\
	Our proofs for Theorems \ref{thm:MainThm1} and \ref{thm:MainThmLip} are based on a conical analogue to Jones' variational integral. First note that on Lipschitz domains, for every $\xi\in\partial D$ there exists an open cone $\Gamma(\xi)\subset D$ with apex at $\xi$ such that if $z\in\Gamma(\xi)$, then $\delta(z)\geq c_{\partial D}|z-\xi|$. Now observe that if Jones' variational integral is finite for some $\xi\in\partial D$ on the boundary of a $C^{1,\mathrm{Dini}}$ domain, the integral 
	\begin{align}\label{eq:ConVar}
		\int_{\Gamma(\xi)}|\nabla u(z)|\delta(z)^{1-d}\mathrm{d}z
	\end{align}	
	is finite as well. Indeed, we integrate over a smaller set and $k(z,\xi)\sim \delta(z)^{1-d}$ for $z\in\Gamma(\xi)$, see Theorem \ref{thm:BogdanLD}. We can now ask whether boundedness of \eqref{eq:ConVar} is achieved on domains more complicated than $C^{1,\mathrm{Dini}}$. The following theorem gives an affirmative answer for Lipschitz domains.
	\begin{theorem}\label{thm:MainThm2}
		Let $D\subset \mathbb R^d$ be a Lipschitz domain and $u\geq 0$ a harmonic function on $D$. Then there exists $\xi\in\partial D$ and a corresponding cone $\Gamma(\xi)$ as described above such that
		\begin{align}\label{eq:LapLip}
			\int_{\Gamma(\xi)}|\nabla u(z)|\delta(z)^{1-d}\mathrm{d}z<\infty
		\end{align}
		Moreover, the set
		\begin{align*}
			\mathcal{B}_c(u, D):=\lbrace\xi\in\partial D: \text{there exists } \Gamma(\xi) \text{ such that }\eqref{eq:LapLip}\text{ holds}\rbrace
		\end{align*} 
		is dense in $\partial D$ and there exists $\gamma\in(0,1]$ such that $\dim(B\cap\mathcal B_c(u,D))\geq d-2+\gamma$ for any ball $B$ with center on $\partial D$. If $D$ is $C^{1,\mathrm{Dini}}$ we can choose $\gamma=1$, i.e. $\mathcal{B}_c(u, D)$ is ultradense in $\partial D$.
	\end{theorem}
	The integral in \eqref{eq:ConVar} is called the \textit{conical variation} of $u$ on the cone $\Gamma(\xi)$ and $\mathcal{B}_c(u, D)$ denotes the set of \textit{Bourgain} points for the conical variation of $u$ on $D$. Simultaneously, Theorem \ref{thm:MainThmLip} provides a strengthening of Theorem 2 in \cite{MuellerRiegler2020}.
	\\
	\par\textbf{Plan of the paper:} In Section \ref{sec:prel} we introduce the necessary preliminaries, while Section \ref{sec:Mainlemma} is devoted to proving the main lemma which encompasses the original ideas of Havin and Mozolyako. Boundedness of the conical variation is then obtained in Section \ref{sec:ConVar} as an application of the main lemma addressed in Section \ref{sec:Mainlemma}. In Section \ref{sec:LPVar} we provide the estimates for Jones' variational integral that prove the Theorems \ref{thm:MainThm1} and \ref{thm:MainThmLip}. We base our proofs on the results obtained in the preceding Section \ref{sec:ConVar}. Finally, Section \ref{sec:GreenEstimates} contains sharp Green function estimates on $C^{1,\mathrm{Dini}}$ domains that are required in the proof of Theorem \ref{thm:MainThm1}.\par
	\textbf{Acknowledgment:} This paper is part of the author's PhD thesis prepared at the Department of Analysis, JKU Linz. The author is supported by the FWF project P34414 "Variational estimates, Green's mappings and Brownian motion". It is with pleasure that the author thanks his supervisor, Paul F.X. Müller,  and his secondary supervisor, Ilia Binder, for many interesting and helpful discussions.
	\section{Preliminaries}\label{sec:prel}
	For $a,b\in\mathbb{R}$, the symbols $a\sim(\lesssim,\gtrsim) b$ denote the usual inequalities with constants that may depend only on $D$. To indicate a specific dependence on some parameter $\rho$, we may write $\sim_\rho$.  We write $\delta(x):=\text{dist}(x,D^c)$. The following inequalities are central to our analysis.
	
	\begin{theorem} [Harnack inequality]
		\label{thm:Harnack_ineq}
		Let $u$ be a nonnegative harmonic function, and $K\subset D$ be a compact set. Then
		\begin{align*}
			\sup_Ku\leq c_K \inf_K u
		\end{align*}
	Moreover if $\mathrm{diam}K\leq N \delta(K)$, then $c_K=c_N$
	\end{theorem} 
	
	\begin{lemma}[Pointwise gradient bound]\label{lem:GradientBound}
		Let $u$ be a nonnegative harmonic function on $D$, then
		\begin{align*}
			|\nabla u(x)|\lesssim\frac{u(x)}{\delta(x)}.
		\end{align*}
	\end{lemma}
	By assumption on $D$, the Dirichlet problem
	\begin{align*}
		Lu=0\quad&\text{on } D\\
		u=\phi\quad&\text{on } \partial D	
	\end{align*}
	is uniquely solvable for any $\phi\in C(\partial D)$. For $x\in D$ the operator $T_x$ mapping $\phi$ to $u(x)$ is linear and bounded by the maximum principle. Thus $T_x$ can be represented by a probability measure on $\partial D$, denoted by $\omega^x$, and the identity
	\begin{align*}
		u(x)=\int_{\partial D}\phi(\xi)\mathrm{d}\omega^x(\xi)
	\end{align*}
	holds for all $x\in D$. We call $\omega^x$ harmonic measure with pole at $x$. 
	\par 
	By approximation, we may take $\phi:=\mathbbm1_A$, for any Borel set $A\subset\partial D$ and obtain that the mapping $x\mapsto\omega^x(A)$ is nonnegative and harmonic, thus by Harnack, for any $x,y\in D$, there exists $c=c(x,y)$ such that $c^{-1}\omega^y(A)\leq\omega^x(A)\leq c\omega^y(A)$. Define 
	\begin{align*}
		k^{y}(x,\xi):=\frac{\mathrm{d}\omega^x(\xi)}{\mathrm d\omega^y(\xi)}
	\end{align*}
	We call $k^{y}(x,\xi)$ the Martin kernel with pole at $y\in D$ and one readily checks that for any $\xi\in\partial D$, $x\mapsto k^y(x,\xi)$ is positive and harmonic.  In the following $k:=k^{z_0}$ for a fixed pole $z_0\in D$.\par 
	A domain $D$ is said to be Lipschitz with character $(r_0,M)$ if for every $x_0\in\partial D$ we can translate and rotate $D$ such that $x_0$ falls onto the origin and 
	\begin{align*}
		D\cap B(0,r_0):=\left\lbrace (x,y)\in\mathbb R^{d-1}\times\mathbb R:y>\Phi(x)\right\rbrace\cap B(0,r_0)
	\end{align*}
	for a Lipschitz function $\Phi:\mathbb R^{d-1}\rightarrow \mathbb R$ with Lipschitz constant $M$.
	For $x\in D$ and $\xi\in\partial D$, we write $x=A_r(\xi)$ to denote that
	\begin{align}\label{eq:ArQ}
		\delta(x)\geq M^{-1}r \quad\text{and}\quad M^{-1}r\leq|x-\xi|\leq r
	\end{align}
	We may always replace the constant $M$ in \eqref{eq:ArQ} by a larger constant $M'$ that only depends on $M$. It is easy to see that for every $r<r_0$ and $\xi\in\partial D$, such a point $A_r(\xi)$ exists. Dahlberg's inequalities \cite{Dahlberg1977} relate harmonic measure and the Green function.
	\begin{lemma}\label{lem:Dahlberg}
		Fix $r<r_0$, $\xi\in\partial D$ and let $z\in D\setminus B(\xi,2r)$, then
		\begin{align}\label{eq:Dahlberg}
			\omega^z(B(\xi,r)\cap \partial D)\sim r^{d-2}G(A_r(\xi),z_0)
		\end{align}
	\end{lemma}
	The next inequality tells us that harmonic measure of a set is large if its pole and the set are close. It originates in Carleson's seminal work \cite{Carleson62} and has been proven to hold for divergence form operators on a broad class of domains, see \cite{Kenig94}.
	\begin{lemma}\label{lem:Carleson} 
		There exists a constant $c=c(D)$ such that	for any $\xi\in\partial D$
		\begin{align}
			\label{eq:Carleson}
			\omega^x(B(\xi,r))\geq c \quad\forall x\in B\left(A_r(\xi),\frac{r}{2}\right)
		\end{align}
	\end{lemma}
	We turn to a very useful estimate for the Green functions. Grüter and Widman \cite{GrueterWidman} showed that if the boundary of $D$ satisfies an exterior cone condition, for fixed $y\in D$ the Green function $x\mapsto G(x,y)$  decays like $\delta(x)^\gamma$ for $x$ close enough the boundary and $\gamma=\gamma(D)\in(0,1]$. Their result even holds for a large class of divergence form operators. In particular the following lemma is true:
	\begin{lemma}\label{lem:TypeSEst}
		Let $D\subset \mathbb R^d$ be a Lipschitz domain and $x,y\in D$. Then there exists $C=C(D)>0$ and $\gamma=\gamma(D)\in(0,1]$ such that 
		\begin{align*}
			G(x,y)\leq C \delta(x)^\gamma|x-y|^{2-d-\gamma}.
		\end{align*}
	\end{lemma}
	In fact, Lemma \ref{lem:TypeSEst} is valid for a much broader class of domains, so-called \textit{domains of type $S$}, which even includes NTA domains \cite{Kenig94}. \par
	We are also very interested in useful lower bounds on the Green function. In \cite{Bogdan2000} Bogdan consolidated and summarized existing estimates for the Green function on Lipschitz domains into a concise form and notably also provided matching lower bounds, see Theorem \ref{thm:Bogdan} in Section \ref{sec:GreenEstimates}. By the identity
	\begin{align}\label{eq:MartinKernelIdentity}
		k(x,\xi)= \lim_{r\rightarrow 0}\frac{G(A_r(\xi),x)}{G(A_r(\xi),z_0)},
	\end{align}
	sharp estimates of the Green function produce sharp estimates of the Martin kernel.	
	\begin{theorem}[\cite{Bogdan2000}]\label{thm:BogdanMartin}
		Let $D$ be bounded Lipschitz domain with character $(r_0,M)$ and let $k$ be the Martin kernel with pole at $z_0\in D$. Then
		\begin{align*}
			k(x,\xi)\sim \frac{G(x,z_0)}{G^2(A(x,\xi),z_0)}|x-\xi|^{2-d}
		\end{align*}
		where $A(x,\xi)=A_{|x-\xi|}(\xi)$ if $|x-\xi|\leq r_0$ and $A(x,\xi)=z_1$ for some $z_1\in D$ with $|z_0-z_1|=\delta(z_0)/2$
	\end{theorem}
	The main ingredients of its proof are boundary Harnack principle and the existence of so called \textit{inverted cones} that allow to cut out singularities of the Green function close to the boundary. They can easily be constructed on Lipschitz domains after passing to local coordinates. Boundary Harnack principle is also available for divergence form operators on NTA domains \cite{JerisonKenig82}, we note that it is possible to prove Bogdan's sharp inequalities in this setting. NTA domains were introduced by Jerison and Kenig \cite{JerisonKenig82} who showed that under, in some sense minimal assumptions on the domain, many results in harmonic analysis/potential theory continue to be valid. The boundary of NTA domains is no longer locally the graph of a function, so the main difficulty presents itself in replacing the notion of \textit{inverted cones} on Lipschitz domains. This is done via Jones' \textit{localization lemma}\cite{LocJones}.\par 
	An estimate that singles out the decay of the Martin kernel was given by Hunt and Wheeden \cite{HuntWheeden68} and later generalized by Jerison and Kenig \cite{JerisonKenig82}.
	\begin{theorem}\label{lem:DecayMartinKernel} Fix $x\in\partial D$, $r<r_0$ and set $\Delta_j:=\Delta_{2^jr}(x)$. Then
		\begin{align}\label{eq:DecayMartinKernel}
			\sup_{\xi\in\Delta_{j+1}\setminus\Delta_j}k(A_r(x),\xi)\leq C\frac{2^{-\alpha j}}{\omega^{z_0}(\Delta_j)}
		\end{align}
	\end{theorem}
	It can be shown, that estimate \eqref{eq:DecayMartinKernel} is obtained when combining Theorem \ref{thm:BogdanMartin} with Lemma \ref{lem:TypeSEst}.
	\\\\
	From now on, we assume that $D$ is a star-shaped Lipschitz domain with the origin as center. Set $z_0=0$. Scaling the domain, we can further assume that $D\subseteq B(0,1)$. Through a covering argument that will be expanded on later, we can lift the special case of star-shaped domains to general Lipschitz domains.\par
	By star-shapedness, there exists a Lipschitz function $\Phi:\mathbb{S}^{d-1}\rightarrow\mathbb{R}^d$ such that $\Phi(\mathbb{S}^{d-1})=\partial D$ and $c_\Phi:=d(\partial D,0)>0$, i.e.: $|\Phi|\geq c_\Phi$.
	\begin{figure}[h!]\label{fig:StarShapeDomain}
		\centering
		\begin{tikzpicture}[scale=1.8,
			every node/.style={font=\small}]
			
			\draw[black, thick]
			(0.2,0.1)
			.. controls (0.1,0.9) and (0.8,1.4) ..
			(1.2,2.0)
			-- (1.35,3.0)
			-- (1.8,4.3)
			.. controls (3.6,4.2) and (5.7,4.0) ..
			(7.2,3.3)
			.. controls (6.8,3.0) and (6.5,2.6) ..
			(6.45,2.4)
			.. controls (6.8,2.0) and (7.1,1.3) ..
			(7.5,1.0)
			-- (6.7,0.9)
			-- (5.8,-0.3)
			-- (3.8,0.2)
			.. controls (2.6,0.3) and (1.5,0.2) ..
			(0.2,0.1);
			
			\coordinate (A) at (0.2,0.1);
			\coordinate (B) at (3.26,2.76);
			\coordinate (C) at (6.45,2.4);
			
			\coordinate (D) at (2.89,3.15);
			\coordinate (E) at (3.6,2.4);
			
			\draw[black, thick]
			(A)--(D)--(E)--cycle;
			\fill[gray!30] (A) -- (D) -- (E) -- cycle;
			
			\draw[black] (A)--($(A)!0.17!(B)$);
			\draw[black] (B)--($(B)!0.78!(A)$);

			\draw[black, thick] (B)--(C);

			\fill[black] (A) circle (1.5pt);
			\fill[black] (B) circle (2pt);
			\fill[black] (C) circle (1.5pt);
			\fill[black] ($(B)!0.7!(C)$) circle (1.5pt);
			
			\node[black] at (0.1,-0.22) {$\xi$};
			
			\node[black] at (2.3,1) {$\Gamma_\beta(\xi)$};
			
			\node[black] at (5.7,2)
			{$x_r=(1-r)x$};
			
			\node[black] at (6.8,2.4)
			{$x$};
			
			\node[black] at (5,0.5) {$D$};
			
			\node[black] at (3.7,2.9) {Origin};
			
			\node[black] at (0.8,0.6) {$\beta$};
		\end{tikzpicture}
		\caption{A star-shaped domain $D$ and the cone $\Gamma_\beta(\xi)$}
	\end{figure}
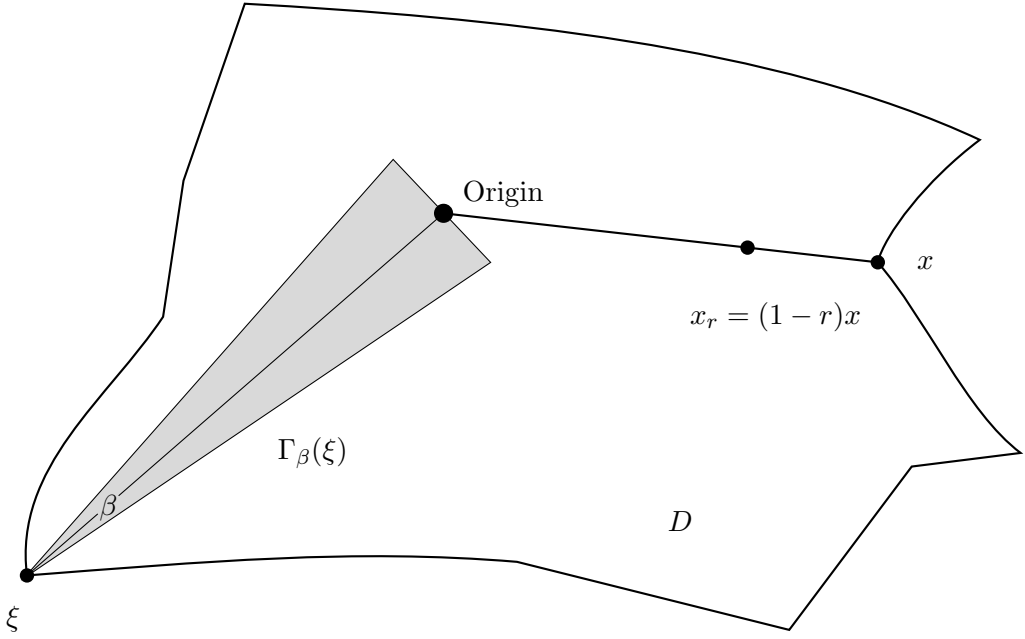
	For $z\in \overline{D}$ and $y\in(0,1)$, let $z_y:=(1-y)z$, see Figure \ref{fig:StarShapeDomain}. Analogously, for $W\subset\mathbb{R}^d$ and $y\in(0,1)$, we define $W_y:=(1-y)W$. Let $F\subseteq\mathbb{R}^d$ and $y\in(0,1)$ such that $W_y\subseteq F$. If $\varphi$ is defined on $F$, we write $\varphi_y$ to denote the function $\varphi_y(x):=\varphi((1-y)x)$, $x\in W$.\par
	
	Fix $\xi\in\partial D$ and write $\Gamma_\beta(\xi)$ to denote the \textit{radial} cone with apex at $\xi$ and height $|\xi|$ that opens toward the origin at an angle $\beta$ around the axis determined by $\xi$ (Figure \ref{fig:StarShapeDomain}). We remark that there exists an angle $\beta_0=\beta(\partial D)$ such that the open cone $\Gamma_{\beta_0}(\xi)$ is contained in $D$ for all $\xi\in\partial D¸$. Thus in particular $d(\xi_y,S)\sim_{\beta_0} y$. Moreover, in the notation introduced previously,
	\begin{align*}
		x_{r}=A_{r}(x)\in D
	\end{align*}
	for $x\in\partial D$ and $r\in(0,1)$. We will make use of the following notational convention unless otherwise stated: If $f$ is a function on $D$ and $y\in(0,1)$, then $f_y:\partial D\rightarrow\mathbb R$ denotes the function defined by $\xi\mapsto f(\xi_y)$. If $p,q:(\partial D)^2\rightarrow\mathbb R$ then $p\circ q:(\partial D)^2\rightarrow[-\infty,\infty]$ formally denotes the function
	\begin{align*}
		(p\circ q)(x,\zeta):=\int_{\partial D}p(x,\xi)q(\xi,\zeta)\mathrm d\omega(\xi).
	\end{align*}  
	We call $p$ as above a \textit{kernel} function. Define the integral operator $P$ corresponding to the kernel $p$ that formally sends a function $f:\partial D\rightarrow\mathbb R$ to the function $(Pf)(x):=(p\circ f)(x)$, $x\in\partial D$. We introduce some algebra on star shaped domains:
	\begin{definition}
		On the set $[0,1]$ we define the operation $\star:[0,1]^2\rightarrow[0,1];\quad a\star b:=1-(1-a)(1-b)=a+b-ab$. $([0,1],\star)$ is a semigroup. 
	\end{definition}
	Notice that for any function $f$ on $D$ and $t,\tau\in(0,1)$, we have $f_t(x_\tau)=f_{t\star\tau}(x)$ for all $x\in \overline D$. By uniqueness of the solution to the Dirichlet problem, the following statements can easily be verified.
	\begin{lemma}\label{lem:MartinKernel}
		For $y\in(0,1)$, set $k_y(x,\xi):=k(x_y,\xi)$, $x,\xi\in\partial D$. The integral operator that corresponds to the kernel $k_y$ is denoted by $K_y$. The following holds true:
		\begin{enumerate}
			\item Let $v$ be a harmonic function on $D$, then for all $x\in \partial D$ and $y_1,y_2\in(0,1)$ we have:
			\begin{align*}
				(K_{y_2}v_{y_1})(x) = v_{y_1\star y_2}(x), \quad \forall x\in\partial D
			\end{align*}
			\item (Semi-group property) For $y_1,y_2\in(0,1)$ we have 
			\begin{align*}
				k_{y_1}\circ k_{y_2} = k_{y_1\star y_2}.	
			\end{align*}
			\item $K_y1=1$.
		\end{enumerate}
	\end{lemma}
	For each $s\in\mathbb N$ we want to find the number $y(s)\in(0,1)$ that satisfies the relation
	\begin{align*}
		y(s)\star2^{-s}=2^{-s+1}
	\end{align*} 
	An easy calculation shows that we have to choose $y(s):=1-\frac{1-2^{-s+1}}{1-2^{-s}}=\frac{2^{-s+1}-2^{-s}}{1-2^{-s}}$. Define the integral kernel
	\begin{align*}
		\tilde{k}_s:=k_{y(s)}=k_{1-\frac{1-2^{-s+1}}{1-2^{-s}}}, \quad s\in\mathbb N
	\end{align*} 
	By the defining property of $y(s)$
	\begin{align*}
		\tilde{k}_{s+1}\circ k_{2^{-s-1}}=k_{2^{-s}}=k_{2^{-s-1}}\circ \tilde k_{s+1}
	\end{align*}
	holds, and of course $\tilde{k}_s\sim k_{2^{-s}}$, by Harnack. 
	\section{The Main Lemma}\label{sec:Mainlemma}
	In this section we present the main tool formulated in Lemma \ref{lem:main_lemma}. The arguments below are inspired by those developed by Havin and Mozolyako for $C^2$ domains \cite{HavinMozol2016}. We base our approach on the essential discrete backbone of their argument. Thus the differential equation, crucial to Havin Mozolyako, disappears in our proof, and resurfaces as a simple discrete telescoping argument below. Moreover, many intermediate limiting processes in the original argument become obsolete.
	\begin{lemma}[Main Lemma]\label{lem:main_lemma}
		Suppose that $D\subset\mathbb R^d$ is a Lipschitz domain and star-shaped with respect to the origin. Let $(c_s)_s$ be a sequence of kernel functions on $(\partial D)^2$ satisfying the following:
		\begin{enumerate}
			\item $|c_s|\leq C2^sk_{2^{-s}}$
			\item $\int_{\partial D}c_s(x,\xi)\mathrm d\omega(\xi)=0$ for all $x\in\partial D$
			\item There exists $\theta\in(0,1]$ such that for each $j>k$ and any nonnegative harmonic function $v$ on $D$ we have 
			\begin{align}\label{eq:c_main_prop}
				\left|\int_{\partial D}c_j(x,\xi)v_{2^{-k}}(\xi)\mathrm d\omega(\xi)\right| \leq C2^{\theta k+(1-\theta)j}v_{2^{-k}}(x)
			\end{align}
		\end{enumerate}
		where $C$ is a constant depending only on $D$.
		Define the kernel $b_s:= k_{2^{-s}}\circ c_s$, $s\in\mathbb N$.
		Then for any nonzero finite measure $\kappa$ supported on $\partial D$ and any $\varepsilon\in(0,\varepsilon(\partial D))$, there exists a finite measure $\nu=\nu_{\varepsilon,\kappa}\neq0$ such that the following holds:\\
		If  $(g_s)_s\subset L^1(\mathrm d\omega)$ is a nonnegative sequence satisfying 
		\begin{align}\label{cond:gs}
			\forall s\in\mathbb{N}:\quad \tilde K_{s+1}g_{s+1}=g_s\quad\text{ and } \quad C_sg_s\geq0,
		\end{align}
		then
		\begin{align*}
			\int_{\partial D} \sum_{s=1}^\infty 2^{-s} B_sg_s\mathrm{d}\nu\lesssim\frac{1}{\varepsilon}\int_{\partial D} K_{1_2}g_1\mathrm d \kappa.
		\end{align*}
		\end{lemma}		
		We divide the proof into several steps. We begin with some preparation. For $\varepsilon>0$ and $s\in\mathbb N$, define the kernel $\pi_{(s,s-1)}:=\pi_{(s,s-1),\varepsilon}:{\partial D}\times {\partial D}\longrightarrow\mathbb{R}$ as
		\begin{align*}
			\pi_{(s,s-1)}:=\tilde k_s-\varepsilon 2^{-s}b_s =\tilde k_s-\varepsilon 2^{-s}k_s\circ c_s.
		\end{align*}
		i.e. we perturb the kernel $\tilde k_s$ by the mean-zero kernel $b_s$. By iterating the above kernels we generate the kernels $\pi_{(s,r)}$, $s>r\geq0$:
		\begin{align*}
			\pi_{(s,r)}:=\pi_{(r+1,r)}\circ\pi_{(r+2,r+1)}\circ\cdots\circ\pi_{(s,s-1)}.
		\end{align*}
		For $r=0$ and $s\geq1$ we write $\pi_{(s,0)}=:\pi_s$. If $\varepsilon>0$ is chosen sufficiently small, the kernels $(\pi_s)_s$ will all be positive. To see this, first observe that by property 1 of the kernel $c_i$ we have
		\begin{align}\label{eq:NaiveIneq}
				\pi_{(i,i-1)}= \tilde k_i -\varepsilon 2^{-i}b_i\geq \tilde k_i -a_1\varepsilon k_{2^{-i}} 
		\end{align}
		for any $i\in\mathbb{N}$, where $a_1$ only depends on ${\partial D}$. Recalling that $\tilde k_i\sim k_{2^{-i}}$, \eqref{eq:NaiveIneq} reads as
		\begin{align}\label{eq:NaiveIneq2}
			\pi_{(i,i-1)}\geq q\tilde k_i
		\end{align}
		where $q=(1-a_2\varepsilon)>0$ for $\varepsilon$ small enough. Thus iterating estimate \eqref{eq:NaiveIneq2} we obtain $\pi_{s}\geq q^{s}k_{1-\frac{1/2}{1 -2^{-s}}}>0$.\par
		Since the kernel $b_s$ has zero mean,
		\begin{align*}
			\int_{\partial D}\pi_{(i,i-1)}(x,\xi)\mathrm d\omega(\xi)=1
		\end{align*}
		holds for any $i\in\mathbb N$. Thus the kernels $\pi_{(s,r)}$ also have mean 1, by Fubini.
		We now state the two key lemmas.
		\begin{lemma}[Difference equation of $\Pi_{s}$]
		\label{lem:diff_eq} The identity
		\begin{align*}
			\Pi_{s+1}g_{s+1}-\Pi_{s} g_s = -\varepsilon 2^{-s-1}\Pi_{s} B_{s+1} g_{s+1}
		\end{align*}
		holds for all $s\in\mathbb N$.
		\end{lemma}
		\begin{proof}
			The proof is a straightforward computation:
			\begin{align*}
				\Pi_{s+1}g_{s+1}-\Pi_{s} g_s &= \Pi_{s}\left(\Pi_{(s+1,s)}g_{s+1}-g_s\right)\\
				&=\Pi_{s}\left(-\varepsilon 2^{-s-1}B_{s+1}g_{s+1}+\tilde K_{s+1}g_{s+1}-g_s\right)\\
				&=-\varepsilon 2^{-s-1} \Pi_{s} B_{s+1} g_{s+1}
			\end{align*}
		\end{proof}
		\begin{lemma}[$\Phi$-property]
			\label{lem:phi-prop}
			Let $\varepsilon\in(0,1)$, $s>r\geq k$, $s,r,k\in\mathbb{N}$ and $v$ nonnegative harmonic on $ D$. Then
			\begin{align*}
				\Pi_{(s,r)}v_{2^{-k}}\sim v_{2^{-k}}.
			\end{align*}
		\end{lemma}
		\begin{proof}
			Let $j\in\mathbb{N}$ such that $r\leq j\leq s$. We have
			\begin{align*}
				|(\Pi_{(j,j-1)}v_{2^{-k}})(x) -v_{2^{-k}}(x)|&\leq\left|(\tilde K_{j}v_{2^{-k}})(x)-v_{2^{-k}}(x)\right|\\
				&+\varepsilon 2^{-j}\int_{\partial D} k_{2^{-j}}(x,\xi)\left|\int_{\partial D}c_j(\xi,\zeta)v_{2^{-k}}(\zeta)\mathrm d \omega(\zeta)\right|\mathrm{d}\omega(\xi):=\Romannum{1}+\Romannum{2}.
			\end{align*}
			Invoking \eqref{eq:c_main_prop}, we estimate
			\begin{align*}
				\Romannum{2}\lesssim\varepsilon2^{\theta(k-j)}\int_{\partial D} k_{2^{-j}}(x,\xi)v_{2^{-k}}(\xi)\mathrm d\omega(\xi)\sim\varepsilon2^{\theta(k-j)}v_{2^{-k}}(x).
			\end{align*}
			By the mean value theorem and the pointwise gradient bound $|\nabla v(x_{2^{-k}})|\lesssim 2^kv(x_{2^{-k}})$, the first term satisfies $\Romannum{1}\lesssim 2^{k-j}v_{2^{-k}}(x)$. In total this yields
			\begin{align*}
				(1-C2^{\theta(k-j)})v_2^{-k}\leq\Pi_{(j,j-1)}v_{2^{-k}}\leq (1+C2^{\theta(k-j)})v_{2^{-k}}.
			\end{align*}
			for some constant $C>0$ only depending on $D$. Observe that now 
			\begin{align*}
				\Pi_{(s,r)}v_{2^{-k}}=\Pi_{(r+1,r)}\cdots\Pi_{(s,s-1)}v_{2^{-k}}&\leq \prod_{j=r+1}^{s}(1+c2^{\theta(k-j)})v_{2^{-k}}\leq C_1 v_{2^{-k}}
			\end{align*}
			where $C_1>0$ only depends on $D$ and $\theta$. For the reverse inequality choose the smallest $m=m(\theta)\geq0$ such that $C2^{-\theta m}<1$. Then, since $k-(r+m)\leq-m$, we have
			\begin{align*}
				\Pi_{(s,r)}v_{2^{-k}}&\geq \prod_{j=r+m}^{s}(1-C2^{\theta(k-j)})\Pi_{(r+1,r)}\cdots\Pi_{(r+m-1,r+m-2)} v_{2^{-k}}.
			\end{align*}
			Using \eqref{eq:NaiveIneq} one can check that 
			\begin{align*}
				\Pi_{(r+1,r)}\cdots\Pi_{(r+m-1,r+m-2)} v_k\geq q^{m-2} v_{1-a}
			\end{align*}
			where $a=\frac{(1-2^{-r})(1-2^{-k})}{1-2^{-r-m+1}}$. Since $1-a\sim2^{-k}$ we have $v_{1-a}\sim v_{2^{-k}}$ and thus
			\begin{align*}
				\Pi_{(s,r)}v_{2^{-k}}\geq\prod_{j=r+m}^{s}(1-C2^{\theta(k-j)})q^m v_{1-a}\geq C_2v_{2^{-k}}
			\end{align*}
		\end{proof}
		We begin the construction of the measure $\nu_{\varepsilon}$. Fix a bounded nonnegative Borel measure $\kappa$ on $\partial D$. The integral operators $(\Pi_s)_{s\in\mathbb N}$ canonically induce functionals $F_s\in C(\partial D)^*$ by the rule $f\mapsto \int_{\partial D}\Pi_sf\mathrm d\kappa$. Since $\Pi_s$ is nonnegative and has mean 1 on $\partial D$ with respect to $\mathrm d\omega$, $\|F_s\|_{C(\partial D)^*}=\kappa(\partial D)$. By the Banach-Alaoglu theorem, the sequence $(F_s)_s$ has a weak* convergent subsequence which we do not relabel for convenience. Let $\nu_{\varepsilon,\kappa}$ be the Borel measure representing this limit. By the above we have 
		\begin{align*}
			\forall f\in C({\partial D}):\quad\int_{\partial D}f\mathrm{d}\nu_{\varepsilon,\kappa}= \lim_{s\rightarrow\infty}\int_{\partial D}\Pi_{s}f\mathrm{d}\kappa.
		\end{align*}
		In particular, $\nu_{\varepsilon,\kappa}(\partial D)=\kappa(\partial D)$. We can now prove Lemma \ref{lem:main_lemma}.
		\begin{proof}[Proof of Lemma \ref{lem:main_lemma}]
			First observe that $B_sg_s\in C({\partial D})$ for all $s\in\mathbb N$. By the definition of $\nu_{\varepsilon,\kappa}$ we get 
			\begin{align}\label{eq:testing}
				\int_{\partial D}\sum_{s=1}^{N}2^{-s}B_sg_s\mathrm{d}\nu_{\varepsilon,\kappa} = \lim_{r\rightarrow\infty}\int_{\partial D}\sum_{s=1}^{N}2^{-s} \Pi_{r}B_sg_s\mathrm{d}\kappa
			\end{align}
			Rewrite $\Pi_r=\Pi_{s-1}\Pi_{(r,s-1)}$. The function $v(z):=(KC_sg_s)(z)$ is nonnegative harmonic on $D$ and $v_{2^{-s}}(x)=( K_sC_sg_s)(x)=(B_sg_s)(x)$ for all $x\in\partial D$. Therefore the $\Phi$-property yields $\Pi_{(r,s-1)}B_sg_s\sim B_sg_s$ and thus by positivity of $\pi_r$ we have $\Pi_{r}B_sg_s\sim \Pi_{s-1}B_sg_s$ if $s\geq2$. Similarly, $\Pi_rB_1g_1\sim\Pi_1B_1g_1$. Substituting this into \eqref{eq:testing}, we obtain
			\begin{align*}
				\int_{\partial D}\sum_{s=1}^{N}2^{-s}B_sg_s\mathrm{d}\nu_{\varepsilon,\kappa}\sim \int_{\partial D}\left[\frac{1}{2}\Pi_1B_1g_1+\sum_{s=2}^{N}2^{-s} \Pi_{s-1}B_sg_s\right]\mathrm{d}\kappa
			\end{align*}
			for any $N>0$. Lemma \ref{lem:diff_eq} shows that 
			\begin{align*}
				\varepsilon2^{-s}\Pi_{s-1}B_sg_s=\Pi_{s-1}g_{s-1}-\Pi_{s} g_s,
			\end{align*}
			Thus, by telescoping 
			\begin{align*}
				\varepsilon\sum_{s=2}^{N}2^{-s} \Pi_{s-1}B_sg_s=\Pi_1g_1-\Pi_{N-1}g_{N-1}
			\end{align*}
			Since $g_{N-1}$ and $\pi_{N-1}$ are nonnegative, $\Pi_{1}g_1-\Pi_{N-1}g_{N-1}\leq\Pi_1g_1$. In combination with the estimates $\pi_1\lesssim k_{1/2}$ and $\pi_1\circ b_1\lesssim k_{1/2}$, we thus obtain
			\begin{align*}
				\int_{\partial D}\sum_{s=1}^{N}2^{-s}B_sg_s\mathrm{d}\nu_{\varepsilon,\kappa}\lesssim\varepsilon^{-1} \int_{\partial D}K_{1/2}g_1\mathrm{d}\kappa.
			\end{align*}
		\end{proof}
		
	\section{Proof of Theorem \ref{thm:MainThm2}}\label{sec:ConVar}
	For the proof, we want to employ Lemma \ref{lem:main_lemma}. We first construct a suitable cone for every $\xi\in\partial D$.  Fix a Whitney decomposition $\mathcal{W}$ of $D$ consisting of axis-parallel dyadic cubes such that the condition
	
	\begin{align}\label{eq:Whitney}
		\forall W\in\mathcal{W}:\quad \text{diam}W\leq d(W,\partial D)\leq 4\text{diam} W
	\end{align}
	is satisfied. Furthermore, if $W\in\mathcal{W}$ intersects $\Omega_s:=\lbrace z\in D:2^{-s}\leq d(z,D^c)\leq 2^{-s+1}\rbrace$, then $W$ has sidelength $2^{-s}$. Fix some  $\alpha\in(0,1)$. By $\Gamma_{\mathrm{dc}}(\xi)$ we denote the discrete cone
	\begin{align*}
		\bigcup_{s\in\mathbb N}\mathcal W_s(\xi)
	\end{align*}
	where
	\begin{align*}
		\mathcal{W}_s(x):=\lbrace W  \in \mathcal{W}: W\cap\Omega_s\neq\emptyset\text{ and }\delta(W)\geq\alpha\delta(W,\xi)\rbrace. 
	\end{align*}
	Observe that $\alpha$ close to $1$ corresponds to a narrow cone and $\alpha$ close to $0$ to a wide cone. Obviously, if $z\in\Gamma_{\mathrm{dc}}(\xi)$, then $\delta(z)\sim_\alpha |z-\xi|\sim_\alpha2^{-s}$ for some $s\in\mathbb N$. We remark moreover that for $\alpha>0$ small enough, the discrete cone $\Gamma_{\mathrm{dc}}(\xi)$ contains the radial cone $\Gamma_{\beta_0}(\xi)$ previously constructed. Recall that $\Gamma_{\beta_0}(\xi)$ denotes the cone with apex $\xi$ that opens up towards the origin at an angle $\beta_0$ around the axis $\xi$ and has height $|\xi|$, see Figure \ref{fig:StarShapeDomain}. With this choice of cones, the tasks lies in establishing finiteness of
	\begin{align}\label{eq:DivergenceConicalEstimate}
		\int_{\Gamma_{\mathrm{dc}}(\xi)}|\nabla u(z)|\delta(z)^{1-d} \mathrm dz\sim\sum_{s\in\mathbb N}\sum_{W\in\mathcal W_s(\xi)}2^{-s(1-d)}\int_W|\nabla u(z)|\mathrm dz
	\end{align}
	for many $\xi\in\partial D$. Clearly the right side of \eqref{eq:DivergenceConicalEstimate} can be written as
	\begin{align}\label{eq:discreteConicalVar}
		\sum_{s=0}^{\infty}2^{-s}\sum_{W\in\mathcal{W}_s(\xi)}\dashint_W|\nabla u(z)|\mathrm{d}z.
	\end{align}
 	We will tackle the problem of finding a Bourgain point for the conical variation $\xi\in\partial D$ by investigating weighted summability of the averages of $|\nabla u|$ on Whitney cubes that approach $\xi$ in a prescribed cone. The next theorem states the main results on star-shaped domains. 
	\begin{theorem}\label{thm:MainThm3}
		Suppose $D$ is star-shaped with respect to the origin and Lipschitz. Then there exists $\xi\in \partial D$ with
		\begin{align*}
			V_c(\xi,u):=\sum_{s=0}^{\infty}2^{-s}\sum_{W\in\mathcal{W}_s(\xi)}\dashint_W|\nabla u(z)|\mathrm{d}z<\infty
		\end{align*}
		Moreover, the set
		\begin{align*}
			\mathcal B_c'(u, D):=\left\lbrace\xi\in\partial D:V_c(\xi,u)<\infty\right\rbrace
		\end{align*}
		is dense in $\partial D$ and there exists $\gamma\in(0,1]$ such that $\dim(B\cap\mathcal B_c'(u,D))\geq d-2+\gamma$ for any ball $B$ with center on $\partial D$. If $D$ is $C^{1,\mathrm{Dini}}$ we can choose $\gamma=1$, i.e. $\mathcal{B}_c'(u, D)$ is ultradense in $\partial D$.
	\end{theorem}
	
	We now present the passage from star-shaped Lipschitz to general bounded Lipschitz domains, i.e. we prove Theorem \ref{thm:MainThm2} using Theorem \ref{thm:MainThm3} 
	\begin{proof}[Proof of Theorem \ref{thm:MainThm2}]
		Let us begin with the case where $D$ is Lipschitz . Note that there exist scalars $t_0,a,b>0$ such that for any $x_0\in\partial D$, we can rotate and translate $\mathbb R^d$ so that $x_0$ is mapped to the origin and the set $D'=D\cap U$ is Lipschitz and star-shaped with center at $(0,bt_0)$, where 
		\begin{align*}
			U=\left\lbrace(x',x_d)\in\mathbb R^{d-1}\times\mathbb R:|x'|\leq t_0,|x_d|\leq at_0\right\rbrace,
		\end{align*}
		Observe that it suffices to consider balls $B$ with radius $r<t_0/4$ whose center lies on $\partial D$. Let $B=B(x_0,r)$ be such a ball and construct the domain $D'$ as above. In particular, $B\subset\partial D\cap\partial D'$. By Theorem \ref{thm:MainThm3}, the set $\mathcal{B}_c(u, D')\cap B$ has at least dimension $d-2+\gamma$. We are done since $\mathcal{B}_c(u, D')\cap B\subset \mathcal{B}_c(u, D)\cap B$.\par
		If $D$ is $C^{1,\mathrm{Dini}}$, we have to construct a domain similar to $D'$, that is $C^{1,\mathrm{Dini}}$ as well. The rest of the argument remains unchanged.  To achieve this we construct the set $U'$ by smoothing the edges of $U$ to make it $C^\infty$ and ensuring that the tangential derivatives of $\partial U'$ and of $\partial D$ coincide on the set $\partial D\cap \partial B(x_0,t_0)$. Thus $D\cap U'$ becomes a star-shaped $C^1$ domain. It is now straightforward to verify the defining conditions \eqref{eq:IntDiniCond} and $\eqref{eq:ExtDiniCond}$. (see Section \ref{sec:GreenEstimates}), i.e. $D\cap U'$ is $C^{1,\mathrm{Dini}}$ and star-shaped, see figure \ref{fig:D'}.
	\end{proof}	
	\begin{figure}\label{fig:D'}
		\centering
		\begin{tikzpicture}
			\draw[thick]
			(-3,-1)
			.. controls (-2.7,-0.8) and (-2.3,-0.3) ..
			(-2,-0.25)
			.. controls (-1.7,-0.2) and (-1.3,0.35) ..
			(-1,0.4)
			.. controls (-0.7,0.30) and (-0.3,-0.30) ..
			(0,-0.25)
			.. controls (0.3,-0.20) and (0.7,0.20) ..
			(1,0.25)
			.. controls (1.3,0.30) and (1.7,-0.2) ..
			(2,-0.25)
			.. controls (2.3,-0.1) and (2.7,0.9) ..
			(3,1);
			
			\draw[thick]
			(-2,-0.25)
			.. controls (-2.5,-0.4) and (-2.35,7.5) ..
			(-2.10,8)
			.. controls (-2,8.3) and (2,8.3)..
			(2.1,8)
			.. controls (2.35,7.5) and (2.5,-0.1)..
			(2,-0.25);
			
			\fill (0,-0.25) circle (2pt);
			\node at (0,0){$x_0$};
			\fill (0,6) circle (2pt);
			\node at (0.4,6){$z_0$};
			
			\draw (2,-0.25) arc[start angle=0,end angle=180,radius=2];
			
			\draw[<->] (0,-0.5)--(2,-0.5);
			\node at (1,-0.75){$t_0$};
			
			\node at (-3,8){$D$};
			\node at (1.5,7.5){$D'$};
			\node at (-3.5,-1){$\partial D$};
		\end{tikzpicture}
		\caption{The star-shaped $C^{1,\mathrm{Dini}}$ domain $D'$ with center $z_0$. }
	\end{figure}
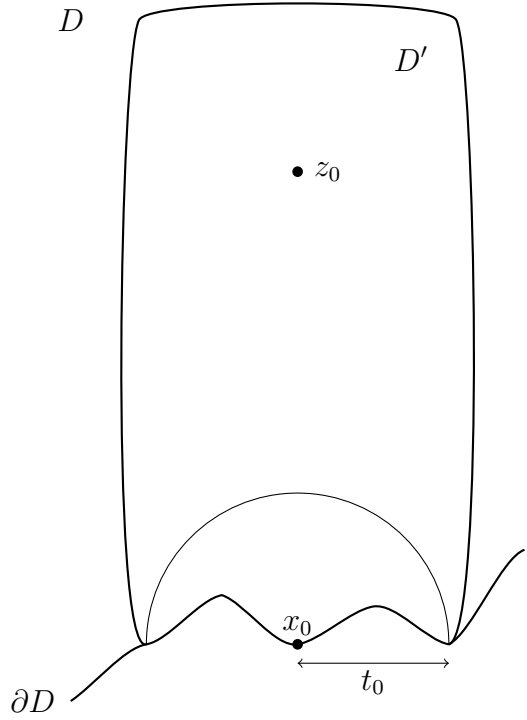
	We prepare the proof of Theorem \ref{thm:MainThm3} by identifying the right kernels $(c_s)_s$ needed to apply Lemma \ref{lem:main_lemma} in order to estimate \eqref{eq:discreteConicalVar}. Set $\rho_s:=1-2^{-s}$. 
	\begin{definition}[Definition of $c_s$] 
		
		For $s\in\mathbb{N}$ we define the kernel $c_s:\partial D\times \partial D\longrightarrow\mathbb{R}$ as
		\begin{align}\label{def:csKernel}
			c_s(x,\xi):=\sum_{W\in\mathcal{W}_s(x)}\dashint_{W}\left\langle(\nabla^1k)(z/\rho_s,\xi),\sigma(z)\right\rangle\mathrm{d}z,
		\end{align}
		where 
		\[
		\sigma(z):=
		\begin{cases}
			\frac{\nabla u(z)}{\|\nabla u(z)\|} & \nabla u(z)\neq0\\
			0 & \nabla u(z)=0
		\end{cases}
		\]
		and $\nabla^1$ denotes differentiation with respect to the first variable $z\in D$. Note, that this kernel depends on the function $u$. We also define
		\begin{align*}
			W_s(x):=\bigcup_{W\in\mathcal{W}_s(x)}W.
		\end{align*}
		Also note that by definition of the Whitney decomposition $|W|=2^{-sd}$ if $W\in\mathcal{W}_s(x)$. Since $|\mathcal{W}_s(x)|\sim_\alpha1$, we thus have $|W_s(x)|\sim2^{-sd}$. Moreover observe that if $W$ is a Whitney cube intersecting $\Omega_s$, scaling it by $\rho_s^{-1}$ keeps $W$ well inside $D$ i.e. $W/\rho_s\subset D$ and $\delta(W/\rho_s)\sim 2^{-s}$. Indeed,
		\begin{align*}
			d(w,\partial D)\geq d(\rho_sw,\partial D)- 2^{-s}|w|\geq (\sqrt{d}-1)2^{-s}.
		\end{align*}
		holds true for any $w\in W/\rho_s$.
		\\\\
		The justification for the particular choice of the kernels $c_s$ is given in the next lemma. In short, the integral operator $C_s$ applied to $u_{2^{-s}}$ reproduces the means of $|\nabla u|$ on the Whitney cubes in $\mathcal{W}_s$ that appear in \eqref{eq:discreteConicalVar}.
	\end{definition}
	\begin{lemma}
		\label{lem:cy_properties}
		For any $\varphi\in L^1(\mathrm d\omega)$ we have
		\begin{align*}
			(C_s\varphi)(x)=\sum_{W\in\mathcal W_s(x)}\dashint_{W}\left\langle(\nabla K\varphi)(z/\rho_s),\sigma(z)\right\rangle \mathrm{d}z
		\end{align*}
		In particular for $\phi(\xi)=u_{2^{-s}}(\xi)$, we have 
		\begin{align*}
			(C_su_{2^{-s}})(x)= \rho_s\sum_{W\in\mathcal W_s(x)}\dashint_{W}\left|\nabla u(z)\right|\mathrm{d}z \sim\dashint_{W_s(x)}|\nabla u(z)|\mathrm{d}z
		\end{align*} 
	\end{lemma}
	\begin{proof}
		 First observe that for any $z\in D$
		\begin{align}
			\label{eq:Cy_prop}
			\begin{aligned}
				\int_{\partial D}\left\langle(\nabla^1k)(z/\rho_s,\xi),\sigma(z)\right\rangle\varphi(\xi)\mathrm{d}\omega(\xi)&=\left\langle\int_{\partial D}(\nabla^1 k)(z/\rho_s,\xi)\varphi(\xi)\mathrm{d}\omega(\xi) ,\sigma(z)\right\rangle\\
				&=\left\langle(\nabla K\varphi)(z/\rho_s) ,\sigma(z)\right\rangle.
			\end{aligned}
		\end{align}
		Using Fubini and \eqref{eq:Cy_prop} we proceed with
		\begin{align*}
			(C_s\varphi)(x)&=\int_{\partial D} c_s(x,\xi)\varphi(\xi)\mathrm{d}\omega(\xi)\\
			&=\sum_{W\in \mathcal W_s(x)}\dashint_{W}\int_{\partial D} \left\langle(\nabla^1k)(z/\rho_s,\xi),\sigma(z)\right\rangle\varphi(\xi)\mathrm{d}\omega(\xi)\mathrm{d}z\\
			&=\sum_{ W\in \mathcal W_s(x)}\dashint_{W}\left\langle(\nabla K\varphi)(z/\rho_s) ,\sigma(z)\right\rangle\mathrm{d}z
		\end{align*}
		
		The choice $\varphi= u_{2^{-s}}$, gives
		\begin{align*}
			\left\langle(\nabla K u_{2^{-s}})(z/\rho_s) ,\sigma(z)\right\rangle=\rho_s\langle\nabla u(z),\sigma(z)\rangle =\rho_s\left|\nabla u(z)\right|
		\end{align*}
		and thus
		\begin{align*}
			(C_s u_s)(x)=\rho_s\sum_{W\in\mathcal W_s(x)}\dashint_{W}\left|\nabla u(z)\right|\mathrm{d}z  .
		\end{align*}
	\end{proof}
	
	 We now verify the conditions of Lemma \ref{lem:main_lemma} for the sequence $(c_s)_s$
	 \begin{lemma}The sequence of kernels $(c_s)_s$ verifies the following three conditions.
	 	\begin{enumerate}
	 		\item$\begin{aligned}
	 			|c_s(x,\xi)|\lesssim 2^sk_{2^{-s}}(x,\xi).
	 		\end{aligned}$
	 		\item$\begin{aligned}
	 			C_s(1)=0.
	 		\end{aligned}$
	 		\item Let $v$ be nonnegative harmonic on $D$ and $j>k$, then
	 		\begin{align*}
	 			\left|\int_{\partial D}c_j(x,\xi)v_{2^{-k}}(\xi)\mathrm d\omega(\xi)\right| \leq C2^kv_{2^{-k}}(x)
	 		\end{align*}
	 	\end{enumerate}
	 \end{lemma}
	\begin{proof}
	 	\begin{enumerate}
	 		\item It is clear that $|c_s(x,\xi)|\lesssim\sup_{z\in W_s(x)}|(\nabla^1 k)(z/\rho_s,\xi)|$. Apply the pointwise gradient bound given in Lemma \ref{lem:GradientBound} and note that since $\delta(W_s(x)/\rho_s)\sim2^{-s}$, Harnack inequality finishes the proof of 1.
	 		\item This is a special case of Lemma \ref{lem:cy_properties} by setting $\phi=1$.
	 		\item By Lemma \ref{lem:cy_properties}, we obtain
	 		\begin{align*}
	 			\left|(C_jv_{2^{-k}})(x)\right|\leq\sum_{ W\in \mathcal W_j(x)}\dashint_{W}|(\nabla v_{2^{-k}})(z/\rho_j)|\mathrm dz
	 		\end{align*}
	 		Since 
	 		\begin{align*}
	 			(\nabla v_{2^{-k}})(z/\rho_j)=\rho_k(\nabla v)\left(\frac{\rho_k}{\rho_j} z\right)
	 		\end{align*}
	 		and $\delta(\rho_k/\rho_j z)\sim2^{-k}$ for $z\in W_j(x)$, we indeed have $|(\nabla v_{2^{-k}})(z/\rho_j)|\lesssim \rho_k2^k v_{2^{-k}}(z/\rho_j)$. Harnack's inequality then readily implies $v_{2^{-k}}(z/\rho_j)\sim v_{2^{-k}}(x)$ for  $z\in W\subset W_j(x)$.
	 	\end{enumerate}
	\end{proof}
	
	Before proving Theorem \ref{thm:MainThm2} we require one more technical Lemma.
	\begin{lemma}\label{lem:tec}
		Let $f\in L^1_{\mathrm{loc}}$ be nonnegative, then
		\begin{align*}
			\int_{\partial D} k(x_{2^{-s}},\xi)\int_{W_s(\xi)}f(z)\mathrm dz\mathrm d\omega(\xi)\geq C \int_{W_s(x)}f(z)\mathrm dz \quad \text{for all }s\in\mathbb N
		\end{align*}
	\end{lemma}
	\begin{proof}
		By Fubini 
		\begin{align*}
			\int_{\partial D} k(x_{2^{-s}},\xi)\int_{W_s(\xi)}f(z)\mathrm dz\mathrm d\omega(\xi)=\int_Df(z)\omega^{x_{2^{-s}}}(V_s(z))\mathrm dz
		\end{align*}
		where $V_s(z):=\lbrace\xi\in {\partial D}: z\in W_s(\xi)\rbrace$. Clearly
		\begin{align*}
			\int_Df(z)\omega^{x_{2^{-s}}}(V_s(z))\mathrm dz\geq\int_{W_s(x)}f(z)\omega^{x_{2^{-s}}}(V_s(z))\mathrm{d}z.
		\end{align*}
 		Now observe that if $z\in W\in\mathcal{W}_s(x)$, then $V_s(z)$ contains $B\left(y,2^{-s-1}\right)\cap\partial D$ where $y\in\partial D$ is such that $\delta(W,y)\leq\delta(W)+2^{-s-1}$. Indeed, let $\xi\in B\left(y,2^{-s-1}\right)\cap\partial D$, then $\delta(W,y)\geq\delta(W,\xi)-2^{-s-1}$ and thus $\delta(W)\geq\delta(W,\xi)-2^{-s}$. For $\xi$ to be in $V_s(z)$ we want to check that $\delta(W)\geq\alpha\delta(W,\xi)$. Due to \eqref{eq:Whitney} this is the case since $\delta(W,\xi)-2^{-s}\geq\alpha\delta(W,\xi)$ by the inequality $2^{-s}\leq d^{-1/2}\delta(W,\xi)$. Indeed, since $\alpha>0$ is small and 
 		\begin{align*}
 			\delta(W,\xi)-2^{-s}\geq (1-d^{-1/2})\delta(W,\xi),
 		\end{align*}
 		we have shown that $\xi\in V_s(z)$. Together with the inequality
 		\begin{align*}
 			|x-y|\leq\text{diam}W+\delta(W,x)+\delta(W,y)\lesssim 2^{-s},
 		\end{align*}
 		Lemma \ref{lem:Carleson} yields
		\begin{align*}
			\omega^{x_{2^{-s}}}(V_s(z))\geq \delta_{\partial D},
		\end{align*}
		thus completing the proof.
	\end{proof}
	We will now prove Theorem \ref{thm:MainThm3}.
	\begin{proof}[Proof of Theorem \ref{thm:MainThm3}]
		Consider the sequence $(g_s)_s$ defined by $g_s:= u_{2^{-s}}$. In light of Lemma \ref{lem:cy_properties} (1), the second condition in \eqref{cond:gs} is immediate. Since $u$ is harmonic, the condition $\tilde K_{s+1}u_{2^{-s-1}}=u_{2^{-s}}$ is satisfied by the definition of $\tilde K_s$. Choosing $\kappa=\omega$, for any $\varepsilon<c_{\partial D}$, the main Lemma \ref{lem:main_lemma} thus yields the existence of a measure $\nu_{\varepsilon}$ satisfying
		\begin{align}\label{eq:NuIntFinite}
			\int_{\partial D}\sum_{s=1}^{\infty}2^{-s} B_su_{2^{-s}}\mathrm d\nu_{\varepsilon}\lesssim \varepsilon^{-1}u(0),
		\end{align}
		that is, the integrand is finite $\nu_{\varepsilon}$ almost everywhere. Since $\nu_\varepsilon\neq0$, there exists $x\in\partial D$ such that
		\begin{align}\label{eq:explicitBRep}
			\sum_{s=1}^{\infty}2^{-s} (B_su_{2^{-s}})(x)=\sum_{s=1}^{\infty}2^{-s}\int_{\partial D} k(x_{2^{-s}},\xi)\dashint_{W_s(\xi)}|\nabla u(z)|\mathrm dz\mathrm d\omega(\xi)
		\end{align}
		is finite. Finally, invoking Lemma \ref{lem:tec} with $f=|\nabla u|$ for every $s\in\mathbb N$ yields the inequalities
		\begin{align*}
			\int_{\partial D} k(x_{2^{-s}},\xi)\dashint_{W_s(\xi)}|\nabla u(z)|\mathrm dz\mathrm d\omega(\xi)\geq C\dashint_{W_s(x)}|\nabla u(z)|\mathrm dz,
		\end{align*}
		thus finishing the proof.
	\end{proof}
	\subsection{Localization Estimates}
	This section is dedicated to providing estimates that allow to narrow down the possible locations of Bourgain points to arbitrary small scales, that is we localize the set $\mathcal B_c'$. We accomplished this by extracting more information on the support of the measures $\nu_\varepsilon$ that are constructed with the kernels $c_s$ given in \eqref{def:csKernel} as input. 
	\begin{lemma}\label{lem:localization}
	 Let $D$ be a star-shaped Lipschitz domain and $B=B(x,r)$ with $r>0$, $x\in\partial D$. If we choose $\kappa=\omega$ in the construction of $\nu_\varepsilon$, then there exists $\varepsilon(B)>0$ such that
		\begin{align*}
			\nu_{\varepsilon}(B\cap {\partial D})>0 \quad\forall\varepsilon<\varepsilon(B)
		\end{align*}
	\end{lemma}
	\begin{proof}
		We write $\Delta_r(x)$ for $B(x,r)\cap\partial D$. Fix some $l\in\mathbb N$ and observe that 
		\begin{align}\label{eq:SplittingBigBall}
			\nu_\varepsilon(\Delta_r(x))&\geq\int_{\Delta_r(x)}\omega^{\xi_{2^{-l}}}(\Delta_{r/2}(x))\mathrm d\nu_{\varepsilon}(\xi) \\
			&=\int_{\partial D}\omega^{\xi_{2^{-l}}}(\Delta_{r/2}(x))\mathrm d\nu_{\varepsilon}(\xi) -\int_{\partial D\setminus \Delta_{r}(x)}\omega^{\xi_{2^{-l}}}(\Delta_{r/2}(x))\mathrm d\nu_{\varepsilon}(\xi)=: \Romannum{1}-\Romannum{2}.
		\end{align}
		We estimate $\Romannum{1}$ from below and $\Romannum{2}$ from above. The constants $a_i>0$ that will appear in the inequalities below only depend on $D$.
		Let $\phi$ denote the nonnegative harmonic function $z\mapsto\omega^z(\Delta_{r/2}(x))$. The Phi-property implies that $\Pi_s\phi_{2^{-l}}\geq a_1\Pi_l\phi_{2^{-l}}$ for all $s>l$. Moreover 
		\begin{align*}
			\pi_l\gtrsim (1-a_2\varepsilon)^lk_{1-\frac{1/2}{1-2^{-l}}}
		\end{align*} 
		and thus
		\begin{align}\label{eq:IEstimate}
			\Romannum{1}=\lim_{s\rightarrow\infty}\int_{\partial D}\Pi_s\phi_{2^{-l}}\mathrm
			d\omega\geq a_2\int_{\partial D}\Pi_l\phi_{2^{-l}}\mathrm d\omega\geq a_3(1-a_2\varepsilon)^l\omega(\Delta_{r/2}(x)).
		\end{align}
		For $\Romannum{2}$ observe that
		\begin{align*}
			\int_{\partial D\setminus \Delta_{r}(x)}\omega^{\xi_{2^{-l}}}(\Delta_{r/2}(x))\mathrm d\nu_{\varepsilon}(\xi)\leq\sup_{\xi\in\partial D\setminus \Delta_{r}(x)}\omega^{\xi_{2^{-l}}}(\Delta_{r/2}(x)),
		\end{align*}
		since $\nu_{\varepsilon}(\partial D\setminus \Delta_r(x))\leq 1$.  We write
		\begin{align}\label{eq:HarmonicMeasIntegral}
			\omega^{\xi_{2^{-l}}}(\Delta_{r/2}(x))=\int_{\Delta_{r/2}(x)}k(\xi_{2^{-l}},\zeta)\mathrm d\omega(\zeta)
		\end{align}
		where $\xi\in \partial D\setminus \Delta_r(x)$ and let $k\in\mathbb N$ be minimal such that $2^{k-l}\geq r$. Since $|\xi-\zeta|\geq r/2\geq 2^{k-2}2^{-l}$, the decay estimates for the Martin kernel in Lemma \ref{lem:DecayMartinKernel} yield the inequality
		\begin{align*}
			k(\xi_{2^{-l}},\zeta)\leq a_4\frac{2^{-\alpha k}}{\omega(\Delta_{2^{k-l}}(\xi))}\leq a_4\frac{r^{-\alpha}2^{-\alpha l}}{\omega(\Delta_r(\xi))},
		\end{align*}
	 	by the choice of $k$. We plug this into \eqref{eq:HarmonicMeasIntegral} and obtain
		\begin{align}\label{eq:IIEstimate}
			\omega^{\xi_{2^{-l}}}(\Delta_{r/2}(x))\leq a_4\frac{\omega(\Delta_{r/2}(x))}{\omega(\Delta_r(\xi))}r^{-\alpha}2^{-\alpha l}.
		\end{align}
		By Dahlberg's inequalities in Lemma \ref{lem:Dahlberg}, $\omega(\Delta_r(\xi))\sim r^{d-2}G(\xi_r,0)$, thus we see that $C(r):=\inf_{\xi\in\partial D}\omega(\Delta_r(\xi))>0$. Applying estimates \eqref{eq:IEstimate} and \eqref{eq:IIEstimate} in \eqref{eq:SplittingBigBall} yields
		\begin{align*}
			\nu_{\varepsilon}(A)\geq \omega(\Delta_{r/2}(x))\left(a_3(1-a_2\varepsilon)^l-a_4r^{-\alpha}C(r)^{-1}2^{-\alpha l}\right)
		\end{align*}
		which is positive if we choose $\varepsilon$ small enough such that $1-a_1\varepsilon>2^{-\alpha}$ and then $l$ large enough. The choice of $l$ only depends on the constants $a_3$, $a_4$, $\varepsilon$ and $r$.
	\end{proof}
	\noindent\textbf{Remark:}
	\begin{itemize}
		\item The content of Lemma \ref{lem:localization} appeared already in \cite[Lemma 3.33(2)]{FMR25}. 
		\item The proof of Lemma \ref{lem:localization} replaces the erroneous estimate in \cite[p.55,line 5]{FMR25}.
	\end{itemize}
	\begin{corollary}
		Let $D$ be a star-shaped Lipschitz domain. Then $\mathcal{B}_c'(u,D)$ is dense in $\partial D$.
	\end{corollary}
	\subsection{Dimension estimates}
	Here, we upgrade density of $\mathcal B_c'(u,D)$ to ultradensity. As a first step, we investigate how the measure $\nu_\varepsilon$ of a ball scales with its radius.
	\begin{lemma}
		Let $D$ be a star-shaped Lipschitz domain. There exists $C=C(D)$ and $\gamma=\gamma(D)\in(0,1]$ such that for any $x\in\partial D$ and $r>0$, the estimate
		\begin{align*}
			\nu_\varepsilon(\Delta_r(x))\leq C r^{d-2+\gamma-C\varepsilon},
		\end{align*}
		holds. Moreover, if $D$ is $C^{1,\mathrm{Dini}}$, we can take $\gamma=1$.
	\end{lemma}	
	\begin{proof}
		Choose $l\in\mathbb N$ such that $2^{-l}\leq r\leq2^{-l+1}$ and consider the nonnegative harmonic function $z\in D\mapsto \phi(z):=\omega^z(\Delta_{3r})=(K\mathbbm1_{\Delta_{3r}})(z)$, where $\Delta_\rho:=\Delta_\rho(x)$. By Carleson's inequality \eqref{eq:Carleson}, $\phi(\xi_{2^{-l}})\geq q=q(D)$ for $\xi\in B$. Thus 
		\begin{align*}
			\nu_\varepsilon(\Delta_r)=\int_{\Delta_r}\mathrm{d}\nu_{\varepsilon}\leq q^{-1}\int_{\partial D}\phi(\xi_{2^{-l}})\mathrm d\nu_{\varepsilon}(\xi)
		\end{align*}
		Recall that by the $\Phi$-property, $\Pi_s\phi_{2^{-l}}=\Pi_l\Pi_{(s,l)}\phi_{2^{-l}}\sim\Pi_l\phi_{2^{-l}}$,  for any $s>l$. We use this in
		\begin{align}\label{eq:NuDef}
			\int_{\partial D}\phi_{2^{-l}}\mathrm d\nu_{\varepsilon}=\lim_{s\rightarrow\infty}\int_{\partial D}\Pi_s\phi_{2^{-l}}\mathrm d\kappa\sim\int_{\partial D} \Pi_l\phi_{2^{-l}}\mathrm d\kappa.
		\end{align}
		Similarly to showing that the kernels $\pi_s$ are positive we can also show that
		\begin{align}\label{eq:NaiveIneqUpper}
			\pi_{l}\lesssim(1+a_1\varepsilon)^l k_{1-\frac{1/2}{1-2^{-l}}},
		\end{align}
		where $a_1$ is a constant only depending on $D$. Notice that $(1+a_1\varepsilon)^l=2^{cl\varepsilon}$ for a suitable $c>0$. Moreover observe that for $x\in\partial D$
		\begin{align}\label{eq:BringToTheCenter}
			\left(K_{1-\frac{1/2}{1-2^{-l}}}\phi_l\right)(\xi)=\int_{\partial D}k\left(\frac{1}{2(1-2^{-s})}\xi,\zeta\right)\phi((1-2^{-s})\zeta)\mathrm d\omega(\zeta)=(K_{1/2}\mathbbm1_{\Delta_{3r}})(\xi)
		\end{align}
		holds. Inserting \eqref{eq:NaiveIneqUpper} and \eqref{eq:BringToTheCenter} into \eqref{eq:NuDef} we see that
		\begin{align}\label{eq:MartinKernelFubini}
				\int_{\partial D}\phi_{2^{-l}}(x)\mathrm d\nu_{\varepsilon}(x)\lesssim 2^{lc\varepsilon}\int_{\partial D}K_{1/2}\mathbbm1_{\Delta_{3r}}\mathrm d\kappa = 2^{lc\varepsilon}\int_{\partial D}\mathbbm1_{\Delta_{3r}}(\xi)\int_{\partial D}k(x_{1/2},\xi)\mathrm{d}\kappa(x)\mathrm d\omega(\xi) 
		\end{align}
		by Fubini. Clearly $\sup_{(\partial D)^2}k_{1/2}<\infty$, thus the right side of \eqref{eq:MartinKernelFubini} is bounded from above by $2^{lc\varepsilon}\omega(\Delta_{3r})$, up to a constant depending on $\kappa(\partial D)$ . By Dahlberg's estimates in Lemma \ref{lem:Dahlberg} and the Green function estimates in Lemma \ref{lem:TypeSEst} we have
		\begin{align}\label{eq:HarmMeasEst}
			\omega(3B)\sim r^{d-2}G(x_{3r},0)\lesssim r^{d-2+\gamma}
		\end{align}
		and thus in total 
		\begin{align*}
			\nu_{\varepsilon}(B)\leq C r^{d-2+\gamma-c\varepsilon}.
		\end{align*}
		If $D$ is $C^{1,\mathrm{Dini}}$, Theorem \ref{thm:BogdanLD} yields that $G(x_{3r},0)\sim r^{d-1}$, i.e. we can take $\gamma=1$ in \eqref{eq:HarmMeasEst}.
	\end{proof}
	With these inequalities at hand, the mass distribution principle yields localized dimension estimates of $\mathcal B_c'(u,D)$. 
	\begin{lemma}
		For all $B(x,r)$, $x\in\partial D$ and $r>0$, the set $B\cap\mathcal B_c'(u, D)$ has Hausdorff dimension at least $d-2+\gamma$. If $D$ is $C^{1,\mathrm{Dini}}$ we can take $\gamma=1$, i.e. $\mathcal B_c'(u, D)$ is ultradense in $\partial D$.
	\end{lemma}
	\begin{proof}
		For $x\in\partial D$ and $r>0$, set $A:=B(x,r)\cap\mathcal B_c'(u,\partial D)$. By Lemma \ref{lem:localization}, choose $\varepsilon>0$ small enough such that $\nu_\varepsilon(\Delta_r(x))>0$. Since $\nu_{\varepsilon}(\partial D\setminus \mathcal {B}_c'(u))=0$, we see that $\nu_\varepsilon(A)=\nu_\varepsilon(\Delta_r(x))>0$. Take any countable covering $(B_j)_j$ of $A$ with radii $r_j>0$ and centers on $\partial D$, then 
		\begin{align}\label{eq:CoveringEst}
			0<\nu_{\varepsilon}(\Delta_r(x))=\nu_{\varepsilon}(A)\leq\sum_j\nu_\varepsilon(B_j)\leq\sum_jr_j^{d-2+\gamma-c\varepsilon}.
		\end{align}
		\eqref{eq:CoveringEst} now readily implies $\dim(A)\geq d-2+\gamma-c\varepsilon$.
	\end{proof}

	\section{Boundedness of Jones' Variational Integral}\label{sec:LPVar}
	It is a standard argument that the integral of a harmonic function on a cone majorizes a weighted integral along the axis of the cone, see e.g. \cite[Chapter \Romannum{4}]{Stein70}. We adapt this consideration to our purpose.
	\begin{lemma}\label{lem:RadVar}
		Let $D$ be a star-shaped Lipschitz domain then there exists $\xi\in\partial D$ such that
		\begin{align}\label{eq:RadialVariation}
			\int_0^1\int_{\partial D}k(x_t,\xi)|\nabla u(\xi_t)|\mathrm d\omega(\xi)\mathrm dt<\infty
		\end{align}
		Moreover, the set $\mathcal B_{\mathrm{rad}}(u, D):=\lbrace\xi\in\partial D: \eqref{eq:RadialVariation} \text{ holds for }\xi\rbrace$ is dense in $\partial D$ and there exists $\gamma\in(0,1]$ such that $\dim(B\cap\mathcal B_{\mathrm{rad}}(u, D))\geq d-2+\gamma$ for any ball $B$ with center on $\partial D$. If $D$ is $C^{1,\mathrm{Dini}}$ we can choose $\gamma=1$, i.e. $\mathcal{B}_{\mathrm{Rad}}(u, D)$ is ultradense in $\partial D$.
	\end{lemma}
	\begin{proof}
		The lemma is proven once we have established the following inequality:
		\begin{align}\label{eq:Balayage}
			\int_0^1\int_{\partial D}k(x_t,\xi)|\nabla u(\xi_t)|\mathrm d\omega(\xi)\mathrm dt\leq C\sum_{s=0}^{\infty}2^{-s}\int_{\partial D}k(x_{2^{-s}},\xi)\dashint_{W_s(\xi)}|\nabla u(z)|\mathrm dz\mathrm d\omega(\xi)
		\end{align}
		Indeed, the right hand side of \eqref{eq:Balayage} is the integrand in \eqref{eq:NuIntFinite} (see \eqref{eq:explicitBRep}), thus the preceding considerations about the measure $\nu_\varepsilon$ imply the localized dimension estimates for the set $\mathcal{B}_{\mathrm{rad}}(u,\partial D)$.\par
		In order to prove \eqref{eq:Balayage}, we begin by observing that the left side in \eqref{eq:Balayage} is bounded from below and above by
		\begin{align}\label{eq:RadialAfterHarnack}
			\sum_{s=0}^{\infty}\int_{\partial D}k(x_{2^{-s}},\xi)\int_{2^{-s}}^{2^{-s+1}}|\nabla u(\xi_t)|\mathrm dt\mathrm d\omega(\xi)
		\end{align}
		due to Harnack. By the mean-value property of the harmonic functions $\partial_iu$ we have
		\begin{align}\label{eq:MeanVal}
			\nabla u(\xi_t)=\dashint_{B_t}\nabla u(z)\mathrm dz
		\end{align}
		where $B_t:=B(\xi_t,t\beta)$. Integrating \eqref{eq:MeanVal} we obtain 
		\begin{align}\label{eq:BallsInCube}
			\int_{2^{-s}}^{2^{-s+1}}|\nabla u(\xi_t)|\mathrm dt\leq\int_{2^{-s}}^{2^{-s+1}}\dashint_{B_t}|\nabla u(z)|\mathrm dz\mathrm dt
		\end{align}
		by the triangle inequality. Note that there exists $N=N(\alpha,\beta)\in\mathbb N$, but independent of $\xi$, such that
		\begin{align*}
			B_t\subset\bigcup_{i=-N}^N W_{s+i}(\xi) \quad\text{for } t\in(2^{-s},2^{-s+1}).
		\end{align*}
		Thus
		\begin{align*}
			\int_{2^{-s}}^{2^{-s+1}}\dashint_{B_t}|\nabla u(z)|\mathrm dz\mathrm dt&\lesssim \int_{2^{-s}}^{2^{-s+1}}t^{-d}\sum_{i=-N}^{N}\int_{W_{s+i}(\xi)}|\nabla u(z)|\mathrm dz\mathrm dt\\
			&\sim 2^{-s} \sum_{i=-N}^{N}\dashint_{W_{s+i}(\xi)}|\nabla u(z)|\mathrm dz
		\end{align*}
		Combining this estimate with \eqref{eq:BallsInCube} and \eqref{eq:RadialAfterHarnack} we obtain inequality \eqref{eq:Balayage}.
	\end{proof}
	\begin{figure}\label{fig:Cone}
		\centering
		\begin{tikzpicture}[scale=6]
			
			\def\H{1}          
			\def\theta{20}     
			\def\yc{0.4}       
			
			\pgfmathsetmacro{\R}{\yc*sin(\theta)}
			
			
			\draw[very thin,dashed] (0,0.5) rectangle (0.5,1);
			\draw[very thin,dashed] (0.5,0.5) rectangle (1,1);
			\node at (1.15,0.75){$W_{s-1}(\xi)$};
			\node at (-0.1,0.5){$2^{-s+1}$};
			
			\foreach \x in {0,0.25,0.5,0.75}
			\draw[very thin,dashed] (\x,0.25) rectangle ++(0.25,0.25);
			\node at (1.12,0.375){$W_{s}(\xi)$};
			\node at (-0.135,0.25){$2^{-s}$};
			
			\foreach \x in {0,0.125,...,0.875}
			\draw[very thin,dashed] (\x,0.125) rectangle ++(0.125,0.125);
			\node at (1.15,0.1875){$W_{s+1}(\xi)$};
			\node at (-0.1,0.125){$2^{-s-1}$};
			
			\foreach \x in {0,0.0625,...,0.9375}
			\draw[very thin, dashed] (\x,0.0625) rectangle ++(0.0625,0.0625);
			
			\draw[very thick] (0.5,0) -- ({0.5+\H*tan(\theta)},\H);
			\draw[very thick] (0.5,0) -- ({0.5-\H*tan(\theta)},\H);

			\filldraw[fill=gray!20,draw=black,thick]
			(0.5,\yc) circle (\R);
			
			\draw[thin] (0.5,0) -- (0.5,\H);
			
			\fill (0.5,\yc) circle (0.3pt);
			
			\node at (0.56,\yc) {$B_t$};
			
		\end{tikzpicture}
		\caption{The cone $\Gamma_{\beta_0}(\xi)$}
	\end{figure}
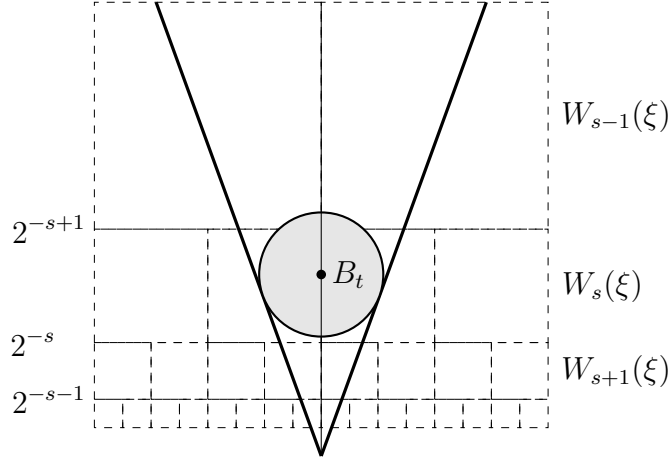
	\noindent\textbf{Remark: }Applying Lemma \ref{lem:RadVar}, subharmonicity of $|\nabla u|$ yields 
	\begin{align}\label{eq:RadInt2}
		\int_0^1|\nabla u(\xi_t)|\mathrm dt<\infty
	\end{align} 
	 for any $\xi\in\mathcal B_{\mathrm{rad}}(u,\partial D)$. Finiteness of \eqref{eq:RadInt2} is the primary subject of interest in \cite{MuellerRiegler2020} and \cite{FMR25} where the content of Lemma \ref{lem:RadVar} is proved using Havin and Mozolyako's original arguments. 
	 \\\\
	We conclude this section by proving the two main Theorems \ref{thm:MainThm1} and \ref{thm:MainThmLip}. Using coarea formula combined with Martin kernel estimates on $C^{1,\mathrm{Dini}}$ domains, we transform Jones' variational integral to the radial integral estimated in Lemma \ref{lem:RadVar}. The localized dimension estimates of Theorem \ref{thm:MainThm3} thus transfer directly.
	\begin{proof}[Proof of Theorem \ref{thm:MainThm1}]:\newline
		\textbf{Step 1:} Assume that $D$ is star-shaped. Define the function $r(z):=|z|/|\Phi(z/|z|)|$ and check that $r^{-1}(\lbrace t\rbrace)=t\partial D$, $t\in(0,1)$. Write $\theta=z/|z|$ and $\Phi=(\phi_j)_j$. Compute 
		\begin{align*}
			\partial_i r(z) = \frac{z_i}{|z||\Phi(\theta)|}-\frac{1}{|\Phi(\theta)|^3}\sum_j\phi_j(\theta)\sum_k(\partial_k\phi_j)(\theta)\left(\delta_{ik}-\frac{z_iz_k}{|z|^2}\right),
		\end{align*}
		so by orthogonality of the two vectors, $|\nabla r(z)|^2\geq1/|\Phi(\theta)|^2$. In total $|\nabla r|\sim_\Phi 1$, by our assumptions on the function $\Phi$.
		Applying the coarea formula yields
		\begin{align}\label{eq:Coarea}
			\begin{aligned}
				\int_{D}|\nabla u(z)|k(z,\xi)|\nabla r(z)|\mathrm{d}z&=\int_0^1\int_{t\partial D} |\nabla u(\zeta)|k(\zeta,\xi)\mathrm{d}\mathcal{H}^{d-1}(\zeta)\mathrm dt\\
				&=\int_0^{1}(1-t)^{d-1}\int_{\partial D}|\nabla u(\zeta_t)|k(\zeta_t,\xi)\mathrm{d}\mathcal{H}^{d-1}(\zeta)\mathrm{d}t
			\end{aligned}
		\end{align}
		Lemma \ref{lem:quasisymm} implies,
		\begin{align}\label{eq:quasisymm}
			k(\zeta_t,\xi)\sim k(\xi_t,\zeta)
		\end{align}
		for $t\leq r_0$. If $t>r_0$, we obtain \eqref{eq:quasisymm} by Harnack's inequality. Plugging \eqref{eq:quasisymm} into the right side of \eqref{eq:Coarea} and replacing surface measure by harmonic measure yields
		\begin{align}\label{eq:RadImpliesFull}
				\int_{D}|\nabla u(z)|k(z,\xi)\mathrm{d}z\sim \int_0^{1}\int_{\partial D}|\nabla u(\zeta_t)|k(\xi_t,\zeta)\mathrm{d}\omega(\zeta)\mathrm{d}t
		\end{align}
		by Lemma \ref{lem:harmeasVShausdorff}. The right side of \eqref{eq:RadImpliesFull} is finite for any $\xi\in\mathcal B_{\mathrm{rad}}(u,D)$, so $\mathcal B_{\mathrm{rad}}(u,D)\subset \mathcal B(u, D)$ implying the desired statements about ultradensity.\\\\
		\textbf{Step 2:} Assume $D$ is not star-shaped. This step is proved by covering a neighborhood of the boundary with finitely many star-shaped $C^{1,\mathrm{Dini}}$ domains and applying Step 1. \par 
		Recall from the proof of Theorem \ref{thm:MainThm2} that there exist scalars $t_0,b>0$ such that for any $x_0\in\partial D$, there exists a star-shaped $C^{1,\mathrm{Dini}}$ domain $D(x_0)\subset D$ satisfying $B(x_0,t_0)\cap D\subset D(x_0)$ and its center $z_0$ satisfies $\delta_D(z_0)\gtrsim bt_0$ (see figure \ref{fig:D'}). Let $B=B(x_0,r)$ be a ball with center $x_0\in\partial D$ and radius $r<t_0/8$. Our goal is to prove $\dim\mathcal B(u,D)\cap B=d-1$. Construct the domain $D(x_0)$ as above and call it $D_0$ with center $z_0$. By Step 1, $\dim \mathcal B(u,D_0)\cap B=d-1$. It suffices to prove that $\mathcal B(u,D_0)\cap B\subset \mathcal B(u, D)\cap B$ i.e. we have to prove that
		\begin{align}
			\forall \xi\in B\cap\partial D:\quad\int_{D_0}|\nabla u(z)| p_{D_0}(z,\xi)\mathrm dz<\infty\Rightarrow\int_{D}|\nabla u(z)|p(z,\xi)\mathrm dz<\infty,
		\end{align}
		where $p_{D_0}$ denotes the local Poisson kernel on the domain $D_0$. Fix some $\xi\in B\cap\partial D$ with
		\begin{align*}
			\int_{D_0}|\nabla u(z)| p_{D_0}(z,\xi)\mathrm dz<\infty.
		\end{align*}
		\textbf{Covering argument: }There exist $c=c(D)>0$ and $x_1,\dots x_N\in\partial D$ satisfying $|x_i-x_0|\geq3/4t_0$ such that
		\begin{align}
			B(x_0,t_0/2)\cup\bigcup_{i=1}^N B(x_i,t_0/2)\supset \lbrace z\in D:\delta(z)<c\rbrace.
		\end{align}
		This can be achieved by covering $\partial D\setminus B(x_0,3/4t_0)$ with finitely many balls of radius $t_0/K$, $K\gg1$. Then take the balls with the same centers but radius $t_0/2$. For each $x_i$ construct the star-shaped $C^{1,\mathrm{Dini}}$ domains $D_i=D(x_i)$ with centers $z_i\in D$, satisfying $\delta_D(z_i)>bt_0$. Note that by their definition, $B_i\cap D\subset D_i$, where $B_i:=B(x_i,t_0/2)$. Applying Step 1 to each of the $D_i$, $i>0$, there exist $\xi_i\in B_i\cap\partial D$ such that
		\begin{align}
			\int_{D_i}|\nabla u(z)|p_{D_i}(z,\xi_i)\mathrm dz<\infty.
		\end{align}
		Note that
		\begin{align}\label{eq:JonesSplitting}
			\int_{D}|\nabla u(z)|p(z,\xi)\mathrm dz&\leq\int_{B_0\cap D}|\nabla u(z)|p(z,\xi)\mathrm dz+ \sum_{i=1}^{N}\int_{B_i\cap D}|\nabla u(z)|p(z,\xi)\mathrm dz\\
			&+\int_{\lbrace\delta(z)\geq c\rbrace}|\nabla u(z)|p(z,\xi)\mathrm dz\label{eq:finiteIntegral}
		\end{align}
		\textbf{Behavior of local Poisson kernels:} We investigate the behavior of the local Poisson kernels $p_{D_i}$ and the global Poisson kernel $p$ on the sets $B_i\cap D$. We begin by bounding the left term in \eqref{eq:JonesSplitting}.
	 	By Theorem \ref{thm:BogdanLD}
		\begin{align}\label{eq:localMartinKernel}
			k^{z_0}_{D_0}(z,\xi)\sim\frac{\delta_{D_0}(z)}{|z-\xi|^d}|z_0-\xi|^{d-1}\sim k^{z_0}_{D}(z,\xi),
		\end{align}
		since $\delta_{D_0}(z)\sim\delta(z)$ for $z\in B_0\cap D$.
		Apply \eqref{eq:localMartinKernel} and recall the factorization 
		\begin{align*}
			p_{D_0}(z,\xi)=p_{D_0}(z_0,\xi) k^{z_0}_{D_0}(z,\xi)\sim p_{D_0}(z_0,\xi)k^{z_0}_{D}(z,\xi) = \frac{p_{D_0}(z_0,\xi)}{p(z_0,\xi)} p(z,\xi).
		\end{align*} 
		Therefore 
		\begin{align}
			\int_{B_0\cap D}|\nabla u(z)|p(z,\xi)\mathrm dz \sim \frac{p(z_0,\xi)}{p_{D_0}(z_0,\xi)} \int_{B_0\cap D}|\nabla u(z)|p_{D_0}(z,\xi)\mathrm dz<\infty
		\end{align}
		
		We now turn to the right term in \eqref{eq:JonesSplitting}. Fix $i=1,\dots,N$, and $z\in B_i\cap D$, then $|z-\xi|\gtrsim|z-\xi_i|$. This follows from the inequalities $|z-\xi_i|\leq t_0$ and 
		\begin{align}
			|z-\xi|\geq|x_i-x_0|-|z-x_i|-|\xi-x_0|\geq t_0/8
		\end{align}
		Moreover $|z_i-\xi|\leq\mathrm{diam}(D)$ and $|z_i-\xi_i|\geq\delta(z_i)\geq b t_0$, so in total $|z_i-\xi|\lesssim|z_i-\xi_i|$ Once again Theorem \ref{thm:BogdanLD} implies
		\begin{align}\label{eq:MartinKernelEstimateD_i}
			k_{D}^{z_i}(z,\xi)\sim\frac{\delta(z)}{|z-\xi|^d}|z_i-\xi|^{d-1}\lesssim \frac{\delta(z)}{|z-\xi_i|^d}|z_i-\xi_i|^{d-1}\sim k_{D_i}^{z_i}(z,\xi_i).
		\end{align}
		We use inequality \eqref{eq:MartinKernelEstimateD_i} in
		\begin{align*}
			p(z,\xi)=p(z_i,\xi) k^{z_i}_{D}(z,\xi)\lesssim p(z_i,\xi)k^{z_i}_{ D_i}(z,\xi_i) = \frac{p(z_i,\xi)}{p_{D_i}(z_i,\xi_i)} p_{D_i}(z,\xi_i),
		\end{align*}
		implying 
		\begin{align}
			\int_{B_i\cap D}|\nabla u(z)|p(z,\xi)\mathrm dz\lesssim \frac{p(z_i,\xi)}{p_{D_i}(z_i,\xi_i)}\int_{B_i\cap D}|\nabla u(z)|p_{D_i}(z,\xi_i)\mathrm dz<\infty.
		\end{align}
		This completes the proof since the integral in \eqref{eq:finiteIntegral} is finite.
	\end{proof}
	\noindent\textbf{Remark:} If $u$ were integrable on $D$, the proof of the second step would follow from splitting Jones' variational integral in two parts: a set where $p$ behaves like the local Poisson kernel and its complement where the decay of $p$ governs the growth of $|\nabla u|$.\\\\
	Practically the same proof as in Step 1 of the preceding proof works for the variant of Jones' variational integral $V_{\mathrm{Lip}}$ on star-shaped Lipschitz domains.
	\begin{proof}[Proof of Theorem \ref{thm:MainThmLip}]
		By the definition of $\mu$ we have
		\begin{align}\label{eq:LipVar}
			\int_D|\nabla u(z)|k(z,\xi)\frac{G(\xi_{\delta(z)},0)}{G(z,0)}\mathrm d\mu(z) =\int_0^1\int_{\partial D}|\nabla u(\zeta_t)|k(\zeta_t,\xi)\frac{G(\xi_t,0)}{G(\zeta_t,0)}\mathrm{d}\omega(\zeta)\mathrm{d}t
		\end{align}
		Applying Lemma \ref{lem:quasisymm} yields that the right side of \eqref{eq:LipVar} is bounded from above and below by
		\begin{align*}
			\int_0^1\int_{\partial D}k(\xi_t,\zeta)|\nabla u(\zeta_t)|\mathrm{d}\omega(\zeta)\mathrm{d}t.
		\end{align*}
		This integral is finite for any $\xi\in\mathcal B_{\mathrm{rad}}(u,\partial D)$.
	\end{proof}

	\section{Green Function Estimates on $C^{1,\mathrm{Dini}}$ Domains}\label{sec:GreenEstimates}
	$C^{1,\mathrm{Dini}}$ domains were first extensively studied by Widman \cite{Widmann67}. They are of particular interest because the regularity of their boundary is essentially minimal for the \textit{boundary point principle} to hold, see \cite{ApuNaz16} for a counterexample if the domain is convex and \cite{ApuNaz22} for an extensive survey on the topic. At the core of the \textit{boundary point principle} or \textit{Hopf-Oleinik lemma} is the inequality
	\begin{align}\label{eq:BPL}
		\liminf_{h\rightarrow0}\frac{u(x_0+hN_{x_0})}{h}>0
	\end{align}
	for any positive harmonic function $u$ and any boundary point $x_0$ with $u(x_0)=0$, where $N_{x_0}$ denotes the unit inward normal at $x_0$. Verifying inequality \eqref{eq:BPL} for the Green function, i.e. $\partial_N G(x,x_0)>0$ for any boundary point $x_0$ and any point $x$ in the domain, already implies the full boundary point principle \cite{ApuNaz22}.\par
	A function $\epsilon:[0,\infty)\rightarrow[0,\infty)$ is called Dini if it is nonnegative, strictly increasing, satisfies $\epsilon(0)=0$, and 
	\begin{align}
		\int_0\frac{\epsilon(t)}{t}\mathrm dt<\infty.
	\end{align}
 	It is convenient to write $x=(x',x_d)\in\mathbb R^d$ where $x_d$ is scalar.
 	
	A $C^1$ domain $D\subset\mathbb R^d$ belongs to the class $C^{1,\mathrm{Dini}}$ if there exists $\rho_0>0$ and a Dini function such that for every $\xi\in\partial D$ we can rotate and translate $\mathbb{R}^d$ so that $\xi$ is mapped to the origin and the relations
	\begin{align}\label{eq:IntDiniCond}
		B(0,\rho_0)\cap\left\lbrace x_d>|x'|\epsilon(|x'|)\right\rbrace\subset D\cap B(0,\rho_0)
	\end{align}
	and
	\begin{align}\label{eq:ExtDiniCond}
		B(0,\rho_0)\cap\left\lbrace x_d<-|x'|\epsilon(|x'|)\right\rbrace\subset D^c\cap B(0,\rho_0)
	\end{align}
	are satisfied. Fix a "center" $z_0\in D$ with $\delta(z_0)\geq2\rho_0$. Hunt \cite{Hunt78} showed that on $C^{1,\mathrm{Dini}}$ domains the inequalities
	\begin{align}\label{eq:HuntIneq}
		C^{-1}\leq\inf_{r>0}\frac{G(A_r(\xi),z_0)}{r}\leq \sup_{r>0}\frac{G(A_r(\xi),z_0)}{r}\leq C,
	\end{align}
	hold for a constant $C>0$ that is uniform in $\xi\in\partial D$. This observation readily implies that 
	\begin{align}\label{eq:GreenIneqOnLDDomains}
		C^{-1}\delta(z)\leq G(z,z_0)\leq C\delta(z)
	\end{align}
	if $|z-z_0|\geq\rho_0/8$ on $C^{1,\mathrm{Dini}}$ domains. Both inequalities in \eqref{eq:HuntIneq} were already proved by Widman \cite{Widmann67}. The lower inequality can be inferred from the proof of the Hopf-Oleinik lemma for $C^{1,\mathrm{Dini}}$ domains in \cite{Widmann67}. Recently, Torres-LaTorre \cite[Theorem 1.1]{Clara26} provided a quantitative lower bound on the size of the quotient in \eqref{eq:BPL}, in terms of the modulus of continuity $\omega$ of the domain's boundary.
	\begin{align*}
		\frac{u(\rho \vec{e_d})}{\rho}\geq \frac{1}{C} \frac{u(r\vec{e_d})}{r}\exp\left(-C\int_\rho^{2r}\frac{\omega(s)}{s}\mathrm ds\right),\quad 0<\rho<r<r_0
	\end{align*}
	In \cite{Bogdan2000}, Bogdan gave a precise description of the contribution of the terms $G(x,z_0)$ and $G(y,z_0)$ to the size of $G(x,y)$ on Lipschitz domains. Assume without loss of generality that $\rho_0\leq \text{diam} D/100$ and fix $z_1\in D$ with $|z_0-z_1|=\rho_0/8$. Bogdan's sharp estimates on the Green function of Lipschitz domains with character $(M,\rho_0)$ are as follows.
	\begin{theorem}\label{thm:Bogdan}
		Let $x,y\notin B(z_0,\rho_0/3)$ and set $\rho:=\max(\delta(x),\delta(y),|x-y|)$. Then
		\begin{align*}
			G(x,y)\sim_D\frac{G(x,z_0)G(y,z_0)}{G^2(A,z_0)}|x-y|^{2-d}
		\end{align*}
		where
		\begin{align*}
			A\in\mathscr{B}(x,y):=
			\begin{cases}
				\left\lbrace A\in D:B(A,M^{-1}\rho)\subset D\cap B(x,3\rho)\cap B(y,3\rho)\right\rbrace,\quad&\rho\leq \rho_0/32\\
				\lbrace z_1\rbrace,\quad&\rho>\rho_0/32
			\end{cases}
		\end{align*}
	\end{theorem}
	Specializing Bogdan's sharp estimates to $C^{1,\mathrm{Dini}}$ domains, we use \eqref{eq:GreenIneqOnLDDomains} to obtain a concise description of the size of $G$. Let $a\wedge b$ denote the minimum of two scalars $a,b$.
	\begin{theorem}\label{thm:BogdanLD}
		Let $D\subset\mathbb R^d$ be a $C^{1,\mathrm{Dini}}$ domain. Then
		\begin{align}\label{eq:GreenFuncEstLD}
			G(x,y)\sim|x-y|^{2-d}\left(\frac{\delta(x)}{|x-y|}\wedge 1\right)\left(\frac{\delta(y)}{|x-y|}\wedge 1\right)
		\end{align}
		and
		\begin{align}\label{eq:MartinKernelDiniDom}
			k^{z_0}(z,\xi)\sim\frac{\delta(z)}{|z-\xi|}\left(\frac{|z-\xi|}{|z_0-\xi|}\right)^{1-d}
		\end{align}
	\end{theorem}
	\begin{proof}The proof of \eqref{eq:GreenFuncEstLD} consists of routinely checking a few cases.
		\begin{enumerate}
			\item Let $x,y\in B(z_0,\rho_0/3)$. It is well known that if $|x-y|\leq N\min\left(\delta(x),\delta(y)\right)$, $G(x,y)\sim_N|x-y|^{2-d}$. Since $\delta(x),\delta(y)\sim_D 1$, this case is done.
			\item Let $x\in B(z_0,\rho_0/3)$ and $y\notin B(z_0,\rho_0)$. Then it is not hard to see that by Harnack $G(x,y)\sim G(x',y)$ where $x'\in B(z_0,\rho_0)\setminus B(z_0,\rho_0/3)$. The next case 3 thus applies.
			\item Let $x,y\notin B(z_0,\rho_0/3)$ and assume wlog that $\delta(x)\leq\delta(y)$.
			\begin{enumerate}
				\item Assume $\rho\geq \rho_0/32$, then by Theorem \ref{thm:Bogdan} the inequalities
				\begin{align*}
					G(x,y)\sim\delta(y)\delta(x)|x-y|^{2-d}
				\end{align*} 
				hold. If $\rho=|x-y|$ then $|x-y|\sim_D 1$ implying \eqref{eq:GreenFuncEstLD}. Now assume $\rho=\delta(y)$, i.e. $\delta(y)\sim_D1$. If $\delta(x)\geq|x-y|$, as in step 1, $G(x,y)\sim|x-y|^{2-d}$, implying \eqref{eq:GreenFuncEstLD}. Else we have $\delta(x)\leq|x-y|\leq\delta(y)$. If $\delta(x)\geq\delta(y)/2$ then $\delta(x)\sim|x-y|\sim\delta(y)\sim 1$, otherwise, if $\delta(x)\leq\delta(y)/2$, then $|x-y|\geq\delta(y)-\delta(x)\geq\delta(y)/2$, i.e. $|x-y|\sim_D1$. In both cases \eqref{eq:GreenFuncEstLD} is satisfied.
				\item Assume $\rho\leq \rho_0/32$. By Theorem \ref{thm:Bogdan}
				\begin{align}\label{eq:InterestingCase}
					G(x,y)\sim_D\frac{\delta(x)\delta(y)}{\delta^2(A)}|x-y|^{2-d},
				\end{align}
				where $\delta (A)\sim\rho$ and $\max(|x-A|,|y-A|)\leq3\rho$. Then
				\begin{enumerate}
					\item Assume $\delta(y)\leq|x-y|=\rho$. We are done since \eqref{eq:InterestingCase} directly implies \eqref{eq:GreenFuncEstLD}.
					\item Assume $|x-y|\leq\delta(y)=\rho$. We further distinguish between two cases:
					\begin{enumerate}
						\item If $\delta(x)\geq|x-y|$ then as in Case 1, $G(x,y)\sim|x-y|^{2-d}$, as desired.
						\item Let $\delta(x)\leq|x-y|$ and observe that \eqref{eq:InterestingCase} reads as
						\begin{align*}
							G(x,y)\sim\frac{\delta(x)}{\delta(y)}|x-y|^{2-d}.
						\end{align*}
						If $\delta(x)\geq\delta(y)/2$ then $\delta(y)\sim |x-y|$. Otherwise $\delta(x)\leq\delta(y)/2$ and thus $|x-y|\geq\delta(y)-\delta(x)\geq\delta(y)$, i.e. again $|x-y|\sim\delta(y)$, which completes the proof.
					\end{enumerate}
				\end{enumerate}
			\end{enumerate}
		\end{enumerate}
		The proof \eqref{eq:MartinKernelDiniDom} is now a matter of inserting the inequalities \eqref{eq:GreenFuncEstLD} into the identity \eqref{eq:MartinKernelIdentity}.
	\end{proof}	
	For $C^{1,1}$ domains, the lower bound in \eqref{eq:GreenIneqOnLDDomains} was proven by Zhao \cite[Theorem 1]{Zhao}, while the upper bound was given by Widman \cite{Widmann67} for $C^{1,\mathrm{Dini}}$ domains. Another consequence of \eqref{eq:GreenIneqOnLDDomains} is that one can switch the variables of the Martin kernel in an appropriate sense.   
	\begin{lemma}\label{lem:quasisymm}
		Let $D$ be a Lipschitz domain with character $(\rho_0,M)$ and let $r<\rho_0$ and $\xi,\zeta\in\partial D$. Then
		\begin{align*}
			\frac{k(A_r(\xi),\zeta)}{k(A_r(\zeta),\xi)}\sim\frac{G(A_r(\xi),0)}{G(A_r(\zeta),0)}.
		\end{align*}
		In particular, if the domain is $C^{1,\mathrm{Dini}}$, $k(A_r(\xi),\zeta)\sim k(A_r(\zeta),\xi)$.
	\end{lemma}
	\begin{proof}
		To begin, observe that $|A_r(\xi)-\zeta|\sim|A_r(\zeta)-\xi|$. We distinguish two cases. 
		\begin{enumerate}
			\item Let $|A_r(\xi)-\zeta|\gtrsim \rho_0$, then by Theorem \ref{thm:BogdanMartin}
			\begin{align*}
				\frac{k(A_r(\xi),\zeta)}{k(A_r(\zeta),\xi)}\sim\frac{G(A_r(\xi),0)}{G(A_r(\zeta),0)}\left(\frac{|A_r(\xi)-\zeta|}{|A_r(\zeta)-\xi|}\right)^{2-d} 
			\end{align*}
			\item	If $|A_r(\xi)-\zeta|\lesssim \rho_0$, Theorem \ref{thm:BogdanMartin} yields
			\begin{align*}
				\frac{k(A_r(\xi),\zeta)}{k(A_r(\zeta),\xi)}\sim\frac{G(A_r(\xi),0)}{G(A_r(\zeta),0)}\left(\frac{G(A_{|A_r(\zeta)-\xi|}(\xi),0)}{G(A_{|A_r(\xi)-\zeta|}(\zeta),0)}\right)^2\left(\frac{|A_r(\xi)-\zeta|}{|A_r(\zeta)-\xi|}\right)^{2-d}
			\end{align*} 
			Note that if the inequality
			\begin{align}\label{eq:HarnackMartinExchange}
				\left|A_{|A_r(\zeta)-\xi|}(\xi)-A_{|A_r(\xi)-\zeta|}(\zeta)\right|\lesssim \min\left(|A_r(\zeta)-\xi|,|A_r(\xi)-\zeta|\right)
			\end{align}
			holds true, the desired result is implied by Harnack's inequality. Adding and substracting $\xi$ and $\zeta$ in the left side of \eqref{eq:HarnackMartinExchange} and performing triangle inequality, the observation 
			\begin{align*}
				|\xi-\zeta|\leq|A_r(\xi)-\xi|+|A_r(\xi)-\zeta|\lesssim \delta(A_r(\xi))+|A_r(\xi)-\zeta|\lesssim |A_r(\xi)-\zeta|
			\end{align*}
			implies inequality \eqref{eq:HarnackMartinExchange}.
		\end{enumerate}
	\end{proof}
	
	We conclude this by summarizing the consequences of the preceding estimates for the Green function on harmonic measure and the Poisson kernel.
	\begin{lemma}\label{lem:harmeasVShausdorff}
		
		Let $D$ be a $C^{1,\mathrm{Dini}}$ domain with $\delta(0)\geq 2\rho_0$. There exists a constant $C=C(D)>0$, such that
		\begin{align}\label{eq:EquivalenceOfMeasures}
			C^{-1}\leq\frac{\omega(A)}{\mathcal{H}^{d-1}(A)}\leq C
		\end{align}
		for any Borel set $A\subset\partial D$.
		In particular, 
		\begin{align*}
			C^{-1}\leq\frac{k(z,\xi)}{p(z,\xi)}\leq C
		\end{align*}
		for all $z\in D$ and $\xi\in\partial D$.
	\end{lemma}
	\begin{proof}
		For any $x\in\partial D$ and $0<r<\rho_0$ we have 
		\begin{align*}
			\omega(B(x,r)\cap\partial D)\sim r^{d-2} G(A_r(x),0)\sim\mathcal{H}^{d-1}(B(x,r))\frac{G(A_r(x),0)}{r}\sim\mathcal{H}^{d-1}(B(x,r)\cap \partial D),
		\end{align*}
		by Dahlberg's inequalities \eqref{eq:Dahlberg}. $\mathcal{A}:=\lbrace A\subset\partial D: A \text{ open } \rbrace$ is a $\cap$-stable generator of the Borel sets on $\partial D$, so by checking that $\mathcal H^{d-1}\sim\omega$ on $\mathcal A$, we obtain $\mathcal H^{d-1}\sim\omega$ on the Borel sets. Thus let $A\in \mathcal{A}$ and let $\lbrace B_j\rbrace$ be a countable collection of open balls with radius smaller than $\rho_0/10$ such that $B_j\subset A$ and $\cup_jB_j=A$. By Vitali's covering theorem we find a subcollection $\lbrace B_{j_i}\rbrace\subset \lbrace B_{j}\rbrace$ of pairwise disjoint balls  such that $\cup_jB_j\subset\cup_i 5B_{j_i}$. Then
		\begin{align*}
			\omega(A)\leq\sum_i\omega(5B_{j_i})\sim\sum_i\mathcal H^{d-1}(5B_{j_i})\sim \sum_i\mathcal H^{d-1}(B_{j_i})=\mathcal{H}^{d-1}\left(\cup_iB_{j_i}\right)\leq\mathcal{H}^{d-1}(A)
		\end{align*}
		The reverse inequality is similar. The second statement follows from \eqref{eq:EquivalenceOfMeasures} and the factorization $p(0,\xi)k(z,\xi)=p(z,\xi)$.		
	\end{proof}
	\newpage
	\printbibliography[heading=bibintoc, title={References}]
	\textsc{J. Fromherz, Institute of Analysis, Johannes Kepler University Linz, 				Altenberger Strasse 69, A-4040 Linz, Austria}
	
	\textit{E-mail address}: jakob.fromherz@jku.at
		
\end{document}